\documentclass{article}

\usepackage{amscd,amsmath,amssymb,amsthm,xcolor,mathrsfs,dsfont,mathtools,setspace}
\newcommand{\R}{\mathbb R}
\newcommand{\N}{\mathbb N}

\newcommand{\J}{{\cal J}}

\newcommand\e{{\rm e}}

\DeclareMathOperator*{\esssup}{ess\; sup}
\DeclareMathOperator*{\essinf}{essinf}
\DeclareMathOperator{\supp}{supp}
\newtheorem{corollary}{Corollary}[section]
\newtheorem{theorem}[corollary]{Theorem}
\newtheorem{lemma}[corollary]{Lemma}
\newtheorem{proposition}[corollary]{Proposition}

\theoremstyle{definition}
\newtheorem{definition}[corollary]{Definition}
\newtheorem{remark}[corollary]{Remark}

\numberwithin{equation}{section}

\begin{document}
\begin{spacing}{1.125}
\title{{\bf On existence and multiplicity of solutions for generalized $(p, q)$--Laplacian equations on unbounded domains}
\footnote{The research that led to the present paper was partially supported 
by MIUR--PRIN project ``Qualitative and quantitative aspects of nonlinear PDEs'' (2017JPCAPN\underline{\ }005), {\sl Fondi di Ricerca di Ateneo} 2017/18 ``Problemi differenziali non lineari''.\\
Both the authors are members of the Research Group INdAM--GNAMPA}}

\author{\sc Addolorata Salvatore\\
{\small Dipartimento di Matematica}\\
{\small Università degli Studi di Bari Aldo Moro} \\
{\small Via E. Orabona 4, 70125 Bari, Italy}\\
{\small \it addolorata.salvatore@uniba.it}\\[15pt]
\sc Caterina Sportelli\\
{\small Department of Mathematics and Statistics}\\
{\small University of Western Australia} \\
{\small 35 Stirling Highway, WA 6009 Crawley, Australia}\\
{\small \it caterina.sportelli@uwa.edu.au}}
\date{March 2023}

\maketitle

\vspace{-.5cm}
\begin{abstract}
This paper deals with the existence and multiplicity of solutions for the generalized $(p, q)$--Laplacian equation
\[
\begin{split}
&-{\rm div}(A(x, u)|\nabla u|^{p-2}\nabla u) +\frac1p A_t(x, u)|\nabla u|^p -{\rm div}(B(x, u)|\nabla u|^{q-2}\nabla u) \\
&\quad\qquad+\frac1q B_t(x, u)|\nabla u|^q + V(x)|u|^{p-2} u+ W(x)|u|^{q-2} u= g(x, u)\quad\qquad\mbox{ in } \R^N,
\end{split}
\]
where $1<q\le p< N$, $A, B:\R^N\times\R\to\R$ are suitable $\mathcal{C}^1$--Carath\'eodory functions with $A_t(x, u)=\frac{\partial A}{\partial t}(x, u), B_t(x, u)=\frac{\partial B}{\partial t}(x, u)$, $V, W:\R^N\to\R$ are proper ``weight functions" and $g:\R^N\times\R\to\R$ is a Carath\'eodory map.\\
Notwithstanding the occurrence of some coefficients which rely upon the solution itself makes the use of variational techniques more challenging, under suitable assumptions on the involved functions, we are able to exploit the variational nature of our problem. In particular, the existence of a nontrivial solution is derived via a generalized version of the Ambrosetti–Rabinowitz Mountain Pass Theorem, based on a weaker version of the classical Cerami--Palais--Smale condition. \\
Finally,  the multiplicity result,  which is thoroughly new also even in the simpler case $q=p$,  is gained under symmetry assumptions and a sharp decomposition of the ambient space.
\end{abstract}

\end{spacing}

\section{Introduction}\label{secintroduction}
In this paper we study the existence of solutions for the generalized $(p, q)$--Laplacian equation
\begin{equation} \label{main}
\begin{split}
&\hspace{-.12in}-{\rm div}(A(x, u)|\nabla u|^{p-2}\nabla u) +\frac1p A_t(x, u)|\nabla u|^p -{\rm div}(B(x, u)|\nabla u|^{q-2}\nabla u) \\
&+\frac1q B_t(x, u)|\nabla u|^q + V(x)|u|^{p-2} u+ W(x)|u|^{q-2} u= g(x, u)\quad\mbox{ in } \R^N,
\end{split}
\end{equation}
where $1<q\le p\le N$, $A, B:\R^N\times\R\to\R$ are $\mathcal{C}^1$--Carath\'eodory functions with $A_t(x, u)=\frac{\partial A}{\partial t}(x, u), B_t(x, u)=\frac{\partial B}{\partial t}(x, u)$, $V, W:\R^N\to\R$ are suitable potentials in a sense discussed later and $g:\R^N\times\R\to\R$ is a Carath\'eodory function. In order to motivate the choice of this problem and highlight the relevance of our results, we present some examples which illustrate the equation covered in this paper.\\
Firstly, we note that if $p=q$ equation \eqref{main} turns into the simpler one
\begin{equation}\label{Ceq}
\begin{split}
-{\rm div}(C(x, u)|\nabla u|^{p-2}\nabla u) &+\frac1p C_t(x, u)|\nabla u|^ p\\
&+ Z(x)|u|^{p-2} u= g(x, u)\quad\mbox{ in } \R^N,
\end{split}
\end{equation}
where we set $C(x, u) =A(x, u)+B(x, u)$ and $Z(x) = V(x) +W(x)$.  The existence of solutions for problem \eqref{Ceq} is related to the existence of solitary waves for the quasilinear Schrödinger equation
\begin{equation}\label{iSchr}
i\partial_t z = -\Delta z - \Delta l(|z|^2) l^{\prime}(|z|^2)z + U(x) z - k(x,|z|)z, 
\end{equation}
with $x\in\R^N$, $t \ge 0$, where the solution $z(x,t)$ is complex in $\R^N\times \R_+$, while $U:\R^N\to\R$, $k:\R^N\times \R_+ \to\R$ and
$l: \R_+ \to \R$ are real functions. Owing to its relevance in several fields of applied sciences, equation \eqref{iSchr} has been widely studied and nowadays still receives a great attention. In fact,  relying to different types of the nonlinear term $l(s)$, it has been derived as a model for different phenomena which come up from plasma physics, fluid mechanics, mechanics and condensed matter theory.  The relation between stationary solutions of \eqref{iSchr} and the model equation \eqref{Ceq} is extensively examined in \cite{CSS2, CSappl}, where a detailed list of further references is proposed, too.\\
Clearly, when $C(x, u)$ is constant,  problem \eqref{Ceq} turns into
\[
-\Delta_p u +Z(x)|u|^{p-2} u =g(x, u)\quad\mbox{ in $\R^N$}
\]
which has been widely investigated in the last decades (see, e.g. \cite{BW, CDS, CePaSo, DiSz, Ra} for the case $p=2$ and \cite{BaGuRo, LZ} for the wider case $p>1$).\\
Moreover, if $p\neq q$, the interest in the study of the equation \eqref{main} is twofold. On the one hand, it is quite challenging from a mathematical viewpoint. On the other hand, equation \eqref{main} has several employments in the applied sciences. 
Observe that if $A(x, u) =B(x, u)\equiv 1$,  \eqref{main} reduces to a classical $(p, q)$--Laplacian equation. In this case, a model of elementary particle physics was studied in \cite{BDF}. 
However, some classical results about $(p,q)$--Laplacian problems in bounded or unbounded domains can be found, for example,  in \cite{AlFi,  Am, BCS, HeLi,MuPa, PaRaRe, PoWa} and references therein.\\
Moreover, we figure it is worthwhile to highlight the correlation with the well known Born-Infeld equation
\begin{equation} \label{bi}
-{\rm div}\left(\frac{\nabla u}{\sqrt{1-\frac{1}{a^2}|\nabla u|^2}}\right) = f(u) \quad\mbox{ in } \R^N,
\end{equation}
with $a\in\R$. 
Equation \eqref{bi} appears quite naturally in several fields such as electromagnetism and in relativity where it represents the mean curvature operator in Lorentz-Minkowski space. 
The relation between \eqref{main} and \eqref{bi} is shown via the first order approximation of the Taylor expansion
\[
\frac{1}{\sqrt{1-x}}= 1+\frac{x}{2}+\frac{3}{8}x^2 +\sum_{k=3}^\infty \binom{k-\frac12}{k} x^k\qquad\mbox{ for } |x|<1,
\]
which makes equation \eqref{bi} turn into
\[
-\Delta u -\frac{1}{2a^2}\Delta_4 u \ = \ f(u) \quad\mbox{ in $\R^N$}
\]
and exhibits a particular case of \eqref{main} with $p=4$, $q=2$, $A(x, u)=\frac{1}{2a^2}$,  $B(x, u)= 1$ and $V(x)=W(x)\equiv 0$. \\
In general,  equation \eqref{main} has been used to model steady--state solutions of reaction--diffusion problems arising in biophysics, in plasma physics and in the study of chemical reaction design. The prototype for these models can be written in the form
\[
-\Delta_p u -\Delta_q +|u|^{p-2} u +|u|^{q-2} u= f(x, u)\quad\mbox{ in $\R^N$},
\]
which originates from a general reaction--diffusion system
\[
u_t = -{\rm div}(D(u)\nabla u) + g(x, u), \qquad \mbox{ with }\quad D(u) =|\nabla u|^{p-2} +|\nabla u|^{q-2}.
\]
In the above mentioned settings, the function $u$ generally stands for a concentration, the term ${\rm div}(D(u)\nabla u)$ corresponds to the diffusion with coefficient $D(u)$, and $g(x, u)$ is the reaction term related to source and loss processes. Typically, in chemical and biological applications, the reaction term $g(x, u)$ is a polynomial of $u$ with variable coefficients (see e.g.  \cite{CI} and references therein).\\
Outlined the range of conceivable applications,  from now on we assume $q<p$ and we consider equation \eqref{main} in its most general formulation. 
As \eqref{main} is settled in $\R^N$, the lack of compactness which occurs makes the classical variational tools difficult to handle.
Here, we do not gain either from symmetry assumptions (see \cite{CaSa2020}) or from concentration compactness arguments (see \cite{HeLi}) but, drawing on the leading work \cite{BF}, we enforce suitable assumption on the potentials $V(x), W(x)$, namely
\begin{equation} \label{miste}
\begin{split}
\essinf_{x\in \R^N}& V(x)>0, \qquad\essinf_{x\in \R^N} W(x)>0,\\
&\lim_{|x|\to +\infty}\int_{B_1(x)}\frac{1}{V(y)}\ dy\ = \ 0,
\end{split}
\end{equation}
where $B_1(x)$ stands for the unitary sphere of $\R^N$ centered in the point $x$. As stated in \cite[Theorem 3.1]{BF}, hypothesis \eqref{miste} assures a compact embedding in suitable weighted Lebesgue spaces on $\R^N$. On the other hand, even if problem \eqref{main} is characterized by the presence of two coefficients which depends on the solution itself,  we give some sufficient conditions for recognizing the variational structure of problem \eqref{main}, so that enquire its solutions shrinks to detect the critical points of the nonlinear functional
\[
\begin{split}
\J(u) &= \frac1p \int_{\R^N} A(x, u)|\nabla u|^p dx +\frac1q \int_{\R^N}B(x, u)|\nabla u|^q dx +\frac1p \int_{\R^N} V(x) |u|^p dx\\
& \quad+\frac1q \int_{\R^N} W(x)|u|^q dx -\int_{\R^N} G(x, u) dx
\end{split}
\]
with $G(x, t) = \int_0^t g(x, s)ds$, in the Banach space
\[
X= W^{1, p}_V(\R^N)\cap W^{1, q}_W(\R^N)\cap L^{\infty}(\R^N).
\]
Among the previously cited problems,  it is also worth noting that our functional $\J$ does not satisfy either the Palais–Smale condition or one of its standard variants in the Banach space $X$. Thus, we are not allowed to use directly existence and multiplicity results as the classical Ambrosetti–Rabinowitz theorems stated in \cite{AR}. Hence,  we need to draw on the pioneer work \cite{CP2} a weaker definition of the Cerami’s variant of Palais–Smale condition (see Definition \ref{wCPS}) which fits our needs.\\
We think that the handling of this definition,  addressed to a wide and challenging framework,  establishes a significant enhancement in this field. In fact, even if we avail of a generalized version of the Mountain Pass Theorem and its symmetric version (Theorems \ref{absMPT} and \ref{SMPT}) already presented in \cite{CPabstract},  we employ both of them in the completely new scenario of a generalized $(p,q)$--Laplacian equation settled in an unbounded domain.\\
Conversly, in view of this framework, we are compelled to assume some hypotheses on the functions $A(x,  t)$ and $B(x, t)$, alongside with the above mentioned \eqref{miste}.  On the one hand,  if we just assume the Carath\'eodory functions $A(x, t)$, $B(x, t)$ are essentially bounded when $t$ is bounded and $g(x, u)$ matches a suitable polynomial growth, we prove that our functional $\J$ is of class $\mathcal{C}^1(X, \R)$.  Thus,
passing through an approximation argument over bounded domains, we prove both the existence of a nontrivial solution and the existence of infinitely many ones, but exploiting also a proper sum decomposition of the weighted Sobolev space $W_V^{1, p}(\R^N)$.  Both of these results require some additional assumptions which we introduce precisely in Section \ref{sec_main}.\\
Now, in order to emphasize the improving of our main results, we present them here in a ``simplified" version. Anyway, we refer the reader to Section \ref{sec_main}, Theorems \ref{ExistMain}, \ref{MoltMain} for a thorough list of all required hypotheses on the concerned functions and a fully detailed statement of these results.
\begin{theorem}
If the potentials $V, W:\R^N\to\R$ satisfy \eqref{miste} and are bounded over bounded sets, if $A, B:\R^N\times\R\to\R$ are $\mathcal{C}^1$--Carath\'eodory functions which properly interact with their partial derivatives and  $g:\R^N\times\R\to\R$ is a Carath\'eodory function which grows no more than $|u|^{p-1} +|u|^{s-1}$ with
\[
p<s<p^*,
\]
has an appropriate behavior in the origin and satisfies the Ambrosetti--\break Rabinowitz condition, thus problem \eqref{main} admits at least one nontrivial weak bounded solution. 
\end{theorem}
It is worthwhile pointing out that the multiplicity result we provide here is entirely novel also in the case $q=p$, i.e., for the equation stated in \eqref{Ceq}.
\begin{theorem} 
In the previous assumptions,  if the requirement about the behaviour of $g$ in the origin is replaced with the assumption that $g(x, \cdot)$ is odd, then problem \eqref{main} has infinitely many weak bounded solutions whose critical levels positively diverge.
\end{theorem}
The paper is organized as follows. In Section \ref{abstractsection} we introduce the weak--Cerami--Palais--Smale condition and some related abstract results to which we refer for the statement of our main results. In Section \ref{sec_var} we give some first hypotheses on the functions $A(x, u), B(x, u), G(x, u)$, on the potentials $V(x), W(x)$ and, in particular, we formulate a variational principle for problem \eqref{main}.
In Section \ref{sec_main} we state our existence and multiplicity results,  provide some further assumptions on the involved functions and test some geometric properties. In Section \ref{sec_bounded} we rephrase problem \eqref{main} over bounded domains, upon which we prove the existence of solutions. Finally, in Section \ref{sec_proof} we prove our main results.

\section{Abstract setting}\label{abstractsection}

Throughout this section, we assume that:
\begin{itemize}
\item $(X, \|\cdot\|_X)$ is a Banach space with dual 
$(X',\|\cdot\|_{X'})$;
\item $(S,\|\cdot\|_S)$ is a Banach space such that
$X \hookrightarrow S$ continuously, i.e., $X \subset S$ and a constant $\sigma_0 > 0$ exists
such that
\[
\|\xi\|_S \ \le \ \sigma_0\ \|\xi\|_X\qquad \hbox{for all $\xi \in X$;}
\]
\item $J : {\mathcal D} \subset S \to \R$ and $J \in C^1(X,\R)$ with $X \subset {\mathcal D}$.
\end{itemize}

In order to avoid any
ambiguity and simplify, when possible, the notation, 
from now on by $X$ we denote the space equipped with
its given norm $\|\cdot\|_X$ while, if the norm $\Vert\cdot\Vert_{S}$ is involved,
we write it explicitly.

For simplicity, taking $\gamma \in \R$, we say that a sequence
$(\xi_n)_n\subset X$ is a {\sl Cerami--Palais--Smale sequence at level $\gamma$},
briefly {\sl $(CPS)_\gamma$--sequence}, if
\[
\lim_{n \to +\infty}J(\xi_n) = \gamma\quad\mbox{and}\quad 
\lim_{n \to +\infty}\|dJ\left(\xi_n\right)\|_{X'} (1 + \|\xi_n\|_X) = 0.
\]
Moreover, $\gamma$ is a {\sl Cerami--Palais--Smale level}, briefly a {\sl $(CPS)$--level}, 
if there exists a $(CPS)_\gamma$ -- sequence.

As $(CPS)_\gamma$ -- sequences may exist which are unbounded in $\|\cdot\|_X$
but converge with respect to $\|\cdot\|_S$,
we have to weaken the classical Cerami--Palais--Smale 
condition in a suitable way according to the ideas already developed in 
previous papers (see, e.g., \cite{CPabstract}).  

\begin{definition} \label{wCPS}
The functional $J$ satisfies the
{\slshape weak Cerami--Palais--Smale 
condition at level $\gamma$} ($\gamma \in \R$), 
briefly {\sl $(wCPS)_\gamma$ condition}, if for every $(CPS)_\gamma$--sequence $(\xi_n)_n$,
a point $\xi \in X$ exists, such that 
\begin{description}{}{}
\item[{\sl (i)}] $\displaystyle 
\lim_{n \to+\infty} \|\xi_n - \xi\|_S = 0\quad$ (up to subsequences),
\item[{\sl (ii)}] $J(\xi) = \gamma$, $\; dJ(\xi) = 0$.
\end{description}
If $J$ satisfies the $(wCPS)_\gamma$ condition at each level $\gamma \in I$, $I$ a real interval, 
we say that $J$ satisfies the $(wCPS)$ condition in $I$.
\end{definition}

By using Definition \ref{wCPS}, the following generalization of the Mountain Pass Theorem can be stated (see \cite[Theorem 1.7]{CPabstract}).

\begin{theorem}[Mountain Pass Theorem] \label{absMPT}
Let $J\in C^1(X,\R)$ be such that $J(0) = 0$
and the $(wCPS)$ condition holds in $\R$.
Moreover, assume that two constants
$\varrho$, $\beta > 0$ and a point $e \in X$ exist such that
\[
u \in X, \; \|u\|_S = \varrho\quad \text{ implies }\quad J(u) \ge \beta,
\]
\[
\|e\|_S > \varrho\qquad\hbox{and}\qquad J(e) < \beta.
\]
Then $J$ has a critical point $u_X \in X$ such that 
\[
J(u_X) = \inf_{\gamma \in \Gamma} \sup_{\sigma\in [0,1]} J(\gamma(\sigma)) \ge \beta
\]
with $\Gamma = \{ \gamma \in C([0,1],X):\, \gamma(0) = 0,\; \gamma(1) = e\}$.
\end{theorem}
Finally, with the stronger assumption that $J$ is symmetric, the following multiplicity result can be provided too (see \cite[Theorem 1.8]{CPabstract}).
\begin{theorem}\label{SMPT}
Let $J\in C^1(X, \R)$ be an even functional such that $J(0)=0$ and $(wCPS)$ holds in $]0, +\infty[$. Assume that $\beta>0$ exists such that
\begin{enumerate}
\item[$(A_{\varrho})$] two closed subspaces $E_{\beta}$ and $Z_{\beta}$ of $X$ exist such that
\[
E_{\beta} + Z_{\beta}=X,\qquad {\rm codim} Z_{\beta} < {\rm dim} E_{\beta}<+\infty,
\]
and $J$ satisfies the following assumptions.
\begin{enumerate}
\item[(i)] A constant $\varrho>0$ exists such that
\[
u\in Z_{\beta}, \; \|u\|_S =\varrho\quad\text{ implies }\quad J(u)\ge\beta;
\]
\item[(ii)] A constant $R>0$ exists such that
\[
u\in E_{\beta}, \; \|u\|_X\ge R\quad\text{ implies }\quad J(u)\le 0,
\]
hence $\displaystyle\sup_{u\in E_{\beta}} J(u)<+\infty$.
\end{enumerate}
\end{enumerate}
Then the functional $J$ possesses at least a pair of symmetric critical points in $X$ whose corresponding critical level belongs to $[\beta, \beta_1]$, with
\[
\beta_1=\displaystyle\sup_{u\in E_{\beta}} J(u) >\beta.
\]
\end{theorem}
\section{Variational setting and preliminary assumptions} \label{sec_var}
Let $\N = \{1,2,\dots\}$ be the set of the strictly positive integers
and, taking any $\Omega$ open subset of $\R^N$ with smooth boundary, $N\ge 2$, we denote by
\begin{itemize}
\item $B_R(x) = \{y\in\R^N : |y-x|< R\}$ the open ball with center in 
$x\in \R^N$ and radius $R>0$;
\item $|D|$ the usual $N$--dimensional Lebesgue measure of a measurable set $D$ in $\R^N$;
\item $d_i$ any strictly positive and each time different constant;
\item $(L^r(\Omega),|\cdot|_{\Omega,r})$ the classical Lebes\-gue space with
norm
\[
|u|_{\Omega,r} = \left(\int_{\Omega}|u|^r dx\right)^{1/r}
\]
if $1\le r<+\infty$;
\item $(L^\infty(\Omega),|\cdot|_{\Omega,\infty})$ the space of the Lebesgue--measurable 
essentially bounded functions endowed with norm $\displaystyle |u|_{\Omega,\infty} = \esssup_{\Omega} |u|$;
\item $W^{1,r}(\Omega)$ and $W_0^{1,r}(\Omega)$ the classical Sobolev spaces 
both equipped with the standard norm 
$\|u\|_{\Omega, r} = (|\nabla u|_{\Omega,r}^r +|u|_{\Omega,r}^r)^{\frac1r}$
 if $1 \le r < +\infty$.
\end{itemize}
Moreover, if $U:\R^N\to\R$ is a measurable function such that
\begin{equation} \label{Uinfess}
\essinf_{x\in\R^N} U(x)>0,
\end{equation}
we denote by
\begin{itemize}
\item $(L_U^r(\Omega),|\cdot|_{\Omega,U,r})$, if $1\le r<+\infty$, the weighted Lebesgue space with
\[
\begin{split}
&L_U^r(\Omega)=\left\{u\in L^r(\Omega): \int_{\Omega} U(x) |u|^r dx <+\infty\right\},\\
&|u|_{\Omega,U,r} =\left(\int_{\Omega} U(x) |u|^r dx\right)^{\frac1r};
\end{split}
\]
\item $W_{U}^{1,r}(\Omega)$ and $W_{0,U}^{1,r}(\Omega)$, if $1\le r<+\infty$, 
the weighted Sobolev spaces
\[
\begin{split}
W_{U}^{1,r}(\Omega) = &\left\{ u\in W^{1,r}(\Omega): \ \int_{\Omega} U(x) |u|^r dx <+\infty\right\},\\
W_{0,U}^{1,r}(\Omega) = &\left\{ u\in W_0^{1,r}(\Omega): \ \int_{\Omega} U(x) |u|^r dx <+\infty\right\}
\end{split}
\]
endowed with the norm
\begin{equation}   \label{weightnorm}
\|u\|_{\Omega,U, r} = (|\nabla u|_{\Omega,r}^r +|u|_{\Omega,U,r}^r)^{\frac1r}.
\end{equation}
\end{itemize}

For simplicity, we put $B_R = B_R(0)$ for the open ball with center at the origin 
and radius $R>0$ and, if $\Omega = \R^N$, we avoid to write the set in the norms, i.e., 
we denote $|\cdot|_{\infty} = |\cdot|_{\R^N,\infty}$ and, for any $1 \le r < +\infty$, we write
\begin{itemize}
\item $|\cdot|_{r} = |\cdot|_{\R^N,r}$ for the norm in $L^r(\R^N)$;
\item $|\cdot|_{U,r} = |\cdot|_{\R^N,U,r}$ for the norm in $L_U^r(\R^N)$;
\item $\|\cdot\|_r = \|\cdot\|_{\R^N, r}$ for the norm in $W^{1,r}(\R^N) = W^{1,r}_0(\R^N)$;
\item $\|\cdot\|_{U, r} = \|\cdot\|_{\R^N,U, r}$ for the norm in $W_U^{1,r}(\R^N) = W_{0,U}^{1,r}(\R^N)$.
\end{itemize}  

\begin{remark} \label{Remb}
If the potential $U(x)$ satisfies assumption \eqref{Uinfess}, then the following
continuous embeddings hold:
\[
L_U^r(\R^N)\hookrightarrow L^r(\R^N) \quad \hbox{for all $1\le r<+\infty$,}
\]
\begin{equation}  \label{WVem}
W_U^{1, r}(\R^N)\hookrightarrow W^{1, r}(\R^N)
\quad \hbox{for all $1\le r<+\infty$.}
\end{equation}
Clearly, \eqref{WVem} holds also replacing $\R^N$ with $\Omega$ any open subset of $\R^N$ with smooth boundary.
\end{remark}

From Remark \ref{Remb} and Sobolev embedding theorems, we deduce the following result (for the compact embedding, see \cite[Theorem 3.1]{BF}).
\begin{theorem}\label{embed}
Let $U:\R^N\to\R$ be a Lebesgue measurable function such that \eqref{Uinfess} holds. 
Then the following continuous embeddings hold.
\begin{itemize}
\item If $l<N$ then
\begin{equation}  \label{cont1}
W_U^{1,l}(\R^N) \hookrightarrow L^r(\R^N) \quad\mbox{ for any }\; l\le r\le \frac{Nl}{N-l};
\end{equation}
\item If $l=N$ then
\begin{equation}  \label{cont2}
W_U^{1, l}(\R^N)\hookrightarrow L^r(\R^N) \quad\mbox{ for any }\; l\le r<+\infty;
\end{equation}
\item If $l>N$ then
\begin{equation}  \label{cont3}
W_U^{1, l}(\R^N)\hookrightarrow L^{r}(\R^N)  \quad\mbox{ for any }\; l\le r\le+\infty.
\end{equation}
\end{itemize}
Furthermore, if also
\[
\int_{B_1(x)}\frac{1}{U(y)} dy\to 0\;\mbox{ as } |x|\to +\infty,
\]
the compact embedding
\begin{equation}    \label{comp}
W_U^{1, l}(\R^N)\hookrightarrow\hookrightarrow L^r(\R^N) \quad\mbox{ for any } l\le r <l^*
\end{equation}
is valid with
\[
l^*=\begin{cases}
\frac{Nl}{N-l} &\hbox{ if } l<N,\\
+\infty &\hbox{ if } l\ge N.
\end{cases}
\]
\end{theorem}
In particular, from Theorem \ref{embed}, it follows that for any $r\ge p$ so that  \eqref{cont1}, respectively \eqref{cont2} or \eqref{cont3},  holds with $r=p$, then a constant $\tau_{r, V}>0$ exists such that
\begin{equation}   \label{sob}
|u|_r\le\tau_{r, V} \|u\|_{V, p} \quad\mbox{ for all } u\in W_V^{1, p}(\R^N).
\end{equation}
From now on, we suppose that
\begin{itemize}
\item[$(P_1)$] the potentials $V, W:\R^N\to\R$ are Lebesgue measurable functions such that
\[
\essinf_{\R^N} V(x)>0,\qquad \essinf_{\R^N} W(x)>0;
\]
\item[$(P_2)$] we have
\[
\int_{B_1(x)}\frac{1}{V(y)} dy\to 0 \qquad\mbox{ as } |x|\to +\infty.
\]
\end{itemize}
Thereby, taking $p, q \in]1, +\infty[$, using the notations introduced at the beginning of this section, we can consider the weighted spaces $(W^{1, p}_V(\R^N), \|\cdot\|_{V})$ and $(W^{1, q}_W(\R^N), \|\cdot\|_W)$.\\
From now on, we set
\begin{equation} \label{Xdefn}
X:= W^{1, p}_V(\R^N)\cap W^{1, q}_W(\R^N)\cap L^{\infty}(\R^N)
\end{equation}
endowed with the norm
\begin{equation} \label{Xnorm}
\|u\|_X := \|u\|_V +\|u\|_W +|u|_{\infty}
\end{equation}
where, to simplify the notations, we denote
\begin{equation}\label{Xnorm2}
\|\cdot\|_{V} = \|\cdot\|_{V, p}\quad\mbox{ and }\quad \|\cdot\|_W=\|\cdot\|_{W, q}.
\end{equation}
Furthermore, Theorem \ref{embed} and definition \eqref{Xdefn} allow us to provide the following result.
\begin{corollary} \label{Cor1}
If assumptions $(P_1)$--$(P_2)$ hold, then
\begin{enumerate}
\item[i)] $(X, \|\cdot\|_X)\hookrightarrow (L^r(\R^N), |\cdot|_r)$ continuously for any $q\le r\le+\infty$;
\item[ii)] $(X, \|\cdot\|_X)\hookrightarrow\hookrightarrow (L^r(\R^N), |\cdot|_r)$ compactly for any $p\le r<+\infty$.
\end{enumerate}
\end{corollary}
\begin{proof}
The embedding $i)$ is trivially satisfied using hypothesis $(P_1)$. On the other hand, if also $(P_2)$ holds, from \eqref{Xdefn} and \eqref{comp} it follows that
\begin{equation}\label{comp1}
X\hookrightarrow\hookrightarrow L^r(\R^N)\quad\mbox{ for any } r\in [p, p^*[.
\end{equation}
Now, let $r\ge p^*$ and $(u_n)_n\subset X$, $u\in X$ such that $u_n\rightharpoonup u$ in $X$. Thus, fixing any $\varepsilon>0$ with $p<p^*-\varepsilon<p^*$, from \eqref{comp1} we have
\[
|u_n -u|_r^r\le |u_n -u|_{\infty}^{r-p^*+\varepsilon}\int_{\R^N}|u_n -u|^{p^*-\varepsilon} dx\to 0
\]
and then $ii)$ is verified.
\end{proof}
We proceed providing a suitable decomposition for the spaces $W_V^{1, p}(\R^N)$ and $X$.\\
We define the subset
\[
\mathcal{S}=\{ v\in W^{1, p}_V(\R^N): |v|_p =1\},
\]
the functional $\Phi: W^{1, p}_V(\R^N)\to\R$ by
\[
\Phi(u) =\|u\|_V^p,
\]
and
\[
\eta_1 =\inf_{v\in\mathcal{S}} \Phi(v)\ge 0.
\]
From \eqref{comp}, there exists $\psi_1\in\mathcal{S}$ such that $\eta_1=\Phi(\psi_1)$.
Thus, starting from $\eta_1$,  the existence of a sequence of positive numbers $(\eta_j)_j$, such that
\begin{equation}\label{inc}
\eta_j\nearrow +\infty\qquad\mbox{ as } j\to +\infty
\end{equation}
and whose corresponding functions are $(\psi_j)_j$, $\psi_j\neq\psi_k$, if $j\neq k$, is established (see, \cite[Section 5]{CP2}). Moreover, they generate the whole space $W^{1, p}_V(\R^N)$. To be more precise, setting $Y_V^j ={\rm span}\{\psi_1, \dots, \psi_j\}$ and denoting by $Z_V^j$ its complement,  for all $j\in\N$ the decomposition
\[
W^{1, p}_V(\R^N) = Y_V^j \oplus Z_V^j 
\]
is valid. Moreover, using \cite[Lemma 5.4]{CP2}, we recall the inequality 
\begin{equation} \label{2Zineq}
\eta_{j+1} |u|_p^p\le\|u\|_V^p \quad\mbox{ for all } u\in Z_V^j.
\end{equation}
Setting
\begin{equation}\label{Zjdefn}
Z^j:= Z_V^j\cap X,
\end{equation}
we have that $Z^j$ is a closed subspace of $X$ of finite codimension. Then a finite dimensional subspace $Y^j$ of $X$ exists which is its topological complement, i.e., the decomposition
\begin{equation}\label{oplus}
X=Y^j\oplus Z^j
\end{equation}
holds. 

We proceed recalling the following definition.
\begin{definition}
A function $h:\R^N\times\R\to\R$ is a $\mathcal{C}^{k}$--Carath\'eodory function, 
$k\in\N\cup\lbrace 0\rbrace$, if
\begin{itemize}
\item $h(\cdot,t) : x \in \R^N \mapsto h(x,t) \in \R$ is measurable for all $t \in \R$,
\item $h(x,\cdot) : t \in \R \mapsto h(x,t) \in \R$ is $\mathcal{C}^k$ for a.e. $x \in \R^N$.
\end{itemize}
\end{definition}

In order to bring out the variational nature of our problem, we introduce some preliminary assumptions. We assume that
\begin{itemize}
\item[$(h_0)$]
$A$ and $B$ are $\mathcal{C}^1$--Carath\'eodory functions;
\item[$(h_1)$] for any $\rho > 0$ we have that
\[
\begin{split}
&\sup_{\vert t\vert\leq \rho} \vert A\left(\cdot, t\right)\vert \in L^{\infty}(\R^N), \ 
\quad \sup_{\vert t\vert\leq \rho} \vert A_t\left(\cdot, t\right)\vert \in L^{\infty}(\R^N)\\
&\sup_{\vert t\vert\leq \rho} \vert B\left(\cdot, t\right)\vert \in L^{\infty}(\R^N), \ 
\quad \sup_{\vert t\vert\leq \rho} \vert B_t\left(\cdot, t\right)\vert \in L^{\infty}(\R^N).
\end{split}
\]
\end{itemize}
Furthermore, we assume that $g:\R^N\times\R\to\R$ exists such that
\begin{itemize}
\item[$(g_0)$] $g(x, t)$ is a $\mathcal{C}^0$--Carath\'eodory function;
\item[$(g_1)$] $a >0$ and $p<s$ exist such that
\[
|g(x, t)|\le a(|t|^{p-1} + |t|^{s-1})\quad \mbox{ a.e. in } \R^N, \mbox{ for all } t\in\R.
\] 
\end{itemize}
\begin{remark}   \label{RemG}
Assumptions $(g_0)$ and $(g_1)$ imply that 
\begin{equation}  \label{Gdefn}
G:(x, t)\in\R^N\times\R\mapsto\int_0^t g(x, s) ds\in\R
\end{equation}
is a well defined $\mathcal{C}^1$--Carath\'eodory function and
\begin{equation}\label{Gle}
|G(x, t)|\le \frac{a}{p}|t|^p +\frac{a}{s}|t|^s \quad \mbox{ a.e. in } \R^N, \mbox{ for all } t\in\R.
\end{equation}
In particular, from \eqref{Gdefn} and $(g_2)$ it follows that
\begin{equation} \label{G0=0}
G(x, 0) =g(x, 0)=0\quad \mbox{ for a.e. } x\in\R^N.
\end{equation}
\end{remark}

Let $u\in X$. Assumption $(h_1)$ and \eqref{Xdefn} ensure that both $A(\cdot, u)|\nabla u(\cdot)|^p\in L^1(\R^N)$ and $B(\cdot, u)|\nabla u(\cdot)|^q\in L^1(\R^N)$. On the other hand,  embedding $i)$ in Corollary \ref{Cor1} and hypotheses $(g_0)$, $(g_1)$ imply that $G(\cdot, u)\in L^1(\R^N)$ too. Thus the functional
\begin{equation}  \label{funct}
\begin{split}
\J(u) &= \frac1p \int_{\R^N} A(x, u)|\nabla u|^p dx +\frac1q \int_{\R^N}B(x, u)|\nabla u|^q dx\\
& \quad +\frac1p \int_{\R^N} V(x) |u|^p dx +\frac1q \int_{\R^N} W(x)|u|^q dx -\int_{\R^N} G(x, u) dx
\end{split}
\end{equation}
is well defined for all $u\in X$.  Furthermore, taking any $u, v\in X$, the G\^ateaux differential of $\J$ in $u$ along the direction $v$ is given by
\begin{equation}   \label{diff}
\begin{split}
&\hspace{-.135in}\langle d\J(u), v\rangle =\int_{\R^N} A(x, u)|\nabla u|^{p-2} \nabla u\cdot\nabla v dx +\frac1p \int_{\R^N} A_t(x, u) v|\nabla u|^p dx\\
&+\int_{\R^N} B(x, u)|\nabla u|^{q-2} \nabla u\cdot\nabla v dx + \frac1q \int_{\R^N} B_t(x, u) v|\nabla u|^q dx\\
&+\int_{\R^N} V(x) |u|^{p-2} u v dx +\int_{\R^N} W(x) |u|^{q-2} u v dx -\int_{\R^N} g(x, u) v dx.
\end{split}
\end{equation}
In particular, the following regularity result can be stated.
\begin{proposition} \label{C1}
Suppose that hypotheses $(P_1)$,  $(h_0)$--$(h_1)$ and $(g_0)$--$(g_1)$ hold. 
If $(u_n)_n\subset X$, $u\in X$ and $M>0$ are such that
\begin{align*} 
&u_n\to u \quad\mbox{ a.e. in } \R^N,\\
&\|u_n -u\|_V \to 0, \quad \|u_n-u\|_W\to 0\quad\mbox{ as } n\to +\infty,\\
&|u_n|_{\infty}\le M \quad\mbox{ for all } n\in\N,
\end{align*}
then
\[
\J(u_n)\to \J(u) \quad\mbox{ and }\quad \|d\J(u_n) -d\J(u)\|_{X^{\prime}}\to 0\quad\mbox{ as } n\to +\infty.
\]
Hence $\J$ is a $\mathcal{C}^1$ functional in $X$ with Fr\'echet differential  
defined as in \eqref{diff}.
\end{proposition}
\begin{proof}
From Definition \eqref{funct},  the desired result occurs by using same arguments as in \cite[Proposition 3.10]{CSS2}.
\end{proof}
\section{Statement of the main results}\label{sec_main}
From now on, besides hypotheses $(P_1)$--$(P_2)$, $(h_0)$--$(h_1)$, $(g_0)$--$(g_1)$ we consider the following additional conditions
\begin{itemize}
\item[$(h_2)$] a constant $\alpha_0>0$ exists such that
\[
\begin{split}
&A(x, t)\ge\alpha_0 \quad\mbox{ a.e. in } \R^N, \mbox{ for all } t\in\R,\\
&B(x, t)\ge\alpha_0 \quad\mbox{ a.e. in } \R^N, \mbox{ for all } t\in\R;
\end{split}
\]
\item[$(h_3)$] some constants $\mu>p$ (in particular it is $\mu>q$, too) and $\alpha_1>0$ exist so that
\[
\begin{split}
(\mu-p)A(x, t)-A_t(x, t)t\ge\alpha_1 A(x, t) \quad\mbox{ a.e. in } \R^N, \mbox{ for all } t\in\R,\\
(\mu-q)B(x, t)-B_t(x, t)t\ge\alpha_1 B(x, t) \quad\mbox{ a.e. in } \R^N, \mbox{ for all } t\in\R;
\end{split}
\]
\item[$(h_4)$] a constant $\alpha_2>0$ exists such that
\[
\begin{split}
pA(x, t) + A_t(x, t)t\ge\alpha_2 A(x, t) \quad\mbox{ a.e. in } \R^N \mbox{ for all } t \in \R,\\
qB(x, t) + B_t(x, t)t\ge\alpha_2 B(x, t) \quad\mbox{ a.e. in } \R^N \mbox{ for all } t \in \R;
\end{split}
\]
\item[$(g_2)$] having $\mu$ as in hypothesis $(h_3)$, then
\[
0<\mu G(x, t)\le g(x, t)t \quad\mbox{ a.e. in } \R^N, \mbox{ for all } t\in\R\setminus\{0\};
\]
\item[$(g_3)$] taking $\alpha_0$ as in assumption $(h_2)$ and $\tau_{p, V}$ as in \eqref{sob}, we have 
\[
\lim_{t\to 0}\frac{g(x, t)}{|t|^{p-2}t} =\bar{\alpha}<\frac{\alpha_0}{\tau_{p, V}^p};
\]
\item[$(P_3)$] for any $\varrho >0$, a constant $C_{\varrho}>0$ exists such that
\[
\esssup_{|x|\le\varrho} V(x)\le C_{\varrho} \quad\mbox{ and }\quad \esssup_{|x|\le\varrho} W(x)\le C_{\varrho}.
\]
\end{itemize}

\begin{remark} \label{RemV1}
If $(h_2)$ holds, then, without loss of generality, we can assume 
$\alpha_0 \le 1$. Moreover, taking $t=0$ in $(h_3)$, we have
also $\mu -p\ge\alpha_1$.
\end{remark}
\begin{remark} \label{Rmk<}
From assumptions $(h_1)$, $(h_3)$ and direct computations we infer that
\[
\begin{split}
&A(x, t)\le a_1 +a_2 |t|^{\mu-p-\alpha_1}\quad\mbox{ a.e. in $\R^N$, for all $t\in\R$},\\
&B(x, t)\le a_1 +a_2 |t|^{\mu-q-\alpha_1}\quad\mbox{ a.e. in $\R^N$, for all $t\in\R$},
\end{split}
\]
where $a_1, a_2$ denote two positive constants.
\end{remark}
\begin{remark}  
We note that \eqref{Gdefn}, assumptions $(g_0)$--$(g_2)$ and straightforward computations 
imply that for any $\varepsilon>0$ a function $\eta_{\varepsilon}\in L^{\infty}(\R^N)$
exists so that $\eta_{\varepsilon}>0$ for a.e. $x\in\R^N$ and
\begin{equation}  \label{Ggeq}
G(x,t)\ge\eta_{\varepsilon}(x) |t|^{\mu} \quad\mbox{ for a.e. } x\in\R^N \hbox{ if } |t|\ge \varepsilon.
\end{equation}
Furthermore,  hypotheses $(g_0)$--$(g_1)$ and $(g_3)$ imply that for any $\varepsilon>0$ a constant $c_{\varepsilon}>0$ exists such that
\[
|g(x, t)|\le(\bar{\alpha} +\varepsilon)|t|^{p-1} +c_{\varepsilon}|t|^{s-1} \quad\mbox{ a.e. in $\R^N$, for all $\in\R$}
\]
and
\begin{equation}\label{Glimcom}
|G(x, t)|\le\frac{\bar{\alpha} +\varepsilon}{p}|t|^p +\frac{c_{\varepsilon}}{s}|t|^s \quad\mbox{ a.e. in $\R^N$, for all $\in\R$}.
\end{equation}
Hence, from \eqref{Gle} and \eqref{Ggeq} it follows that
\begin{equation}\label{p<q}
1<q\le p<\mu \le s.
\end{equation}
\end{remark}
We provide our main existence and multiplicity results.
\begin{theorem} \label{ExistMain}
Suppose that assumptions $(P_1)$--$(P_3)$, $(h_0)$--$(h_4)$ and $(g_0)$--$(g_3)$ hold with
\begin{equation} \label{hpsotto}
s<p^*.
\end{equation}
Then problem \eqref{main} admits at least one weak nontrivial bounded solution. 
\end{theorem}
\begin{theorem} \label{MoltMain}
Under the same hypotheses of Theorem \ref{ExistMain}, but replacing assumption $(g_3)$ with the requirement that $g(x, \cdot)$ is odd,  problem \eqref{main} has a sequence $(u_k)_k$ of weak bounded solutions such that $\J(u_k)\nearrow +\infty$.

\end{theorem}
\begin{remark}
If $q=p$, Theorem \ref{ExistMain} has been already proved (see \cite[Theorem 4.4]{CSS2}). On the other hand, the multiplicity result stated in Theorem \ref{MoltMain} is completely new both in the cases $q=p$ and $q\neq p$. 
\end{remark}
\begin{remark}
Taking $\mu$ as in $(h_3)$, $(g_2)$ and $s$ as in $(g_1)$, it follows by \eqref{p<q} and \eqref{hpsotto} that Theorems \ref{ExistMain} and \ref{MoltMain} hold with
\[
1<q \le p<\mu \le s<p^*.
\]
\end{remark}
Under the assumptions of Theorems \ref{ExistMain} and \ref{MoltMain}, the following geometrical conditions can be provided.
\begin{proposition} \label{geo1}
Assume that hypotheses $(h_0)$--$(h_1)$, $(h_3)$, $(g_0)$--$(g_2)$ and $(P_1)$ hold. Thus for any finite dimensional subspace $F\subset X$, we have
\begin{equation} \label{J-inf}
\lim_{\substack{v\in F\\ \|v\|_X\to +\infty}} \J(v) = -\infty.
\end{equation}
\end{proposition}
\begin{proof}
Suppose by contradiction that a finite dimensional subspace $F\subset X$ exists which does not serve condition \eqref{J-inf}, namely $(u_n)_n\subset F$ can be exhibited so that
\begin{equation} \label{toinf}
\|u_n\|_X\to +\infty\quad\mbox{ as } n\to +\infty
\end{equation}
and for some $M>0$ it is $\J(u_n)\ge -M$, for all $n\in\N$. Setting $v_n:=\frac{u_n}{\|u_n\|_X}$ we have that $\|v_n\|_X=1$ and a subsequence exists such that $v_n\rightharpoonup v$ weakly in $X$. Actually, since $F$ is a finite dimensional subspace, it follows that $v_n\to v$ strongly in $X$ and almost everywhere in $\R^N$. Thus, in particular, we have $\|v\|_X=1$ and, setting $A:=\{x\in\R^N : v(x)\neq 0\}$, we have $|A|>0$. Now, for a.e. $x\in A$, we have $\displaystyle\lim_n |v_n(x)| =|v(x)|>0$ which, together with \eqref{toinf}, ensures that $|u_n(x)|\to +\infty$.  Hence, from \eqref{Ggeq} we obtain that
\begin{equation} \label{G+inf}
\lim_n\frac{G(x, u_n(x))}{|u_n(x)|^{\mu-\alpha_1}} |v_n(x)|^{\mu-\alpha_1} = +\infty \quad\mbox{ for a.e. } x\in A,
\end{equation}
where $\mu-\alpha_1\ge p$ (see Remark \ref{RemV1}). Now, since \eqref{toinf} holds, we may suppose $\|u_n\|_X\ge 1$ for all $n\in\N$. Thus Remark \ref{Rmk<}, \eqref{Xnorm} and straightforward computations imply that
\begin{equation} \label{conti}
\begin{split}
&\frac1p \int_{\R^N} (A(x, u_n)|\nabla u_n|^p + V(x) |u_n|^p) dx\\
&\quad + \frac1q \int_{\R^N} (B(x, u_n)|\nabla u_n|^q dx + W(x) |u_n|^q dx)\\
&\quad\le \frac1p\left(a_1\|u_n\|_V^p +a_2|u_n|_{\infty}^{\mu-p-\alpha_1}\int_{\R^N}|\nabla u_n|^p dx\right)\\
&\quad+\frac1q\left(a_1\|u_n\|_W^q +a_2|u_n|_{\infty}^{\mu-q-\alpha_1}\int_{\R^N}|\nabla u_n|^q dx\right)\\
&\quad\le \frac1p (a_1\|u_n\|_X^p + a_2\|u_n\|_X^{\mu-\alpha_1})+\frac1q (a_1\|u_n\|_X^q + a_2\|u_n\|_X^{\mu-\alpha_1})\\
&\quad\le \frac2q (a_1+a_2)\|u_n\|_X^{\mu-\alpha_1}.
\end{split}
\end{equation}
Summing up,  from \eqref{funct}, \eqref{conti}, hypothesis $(g_2)$ and Fatou's Lemma we get
\[
\begin{split}
0&=\lim_n \frac{-M}{\|u_n\|_X^{\mu-\alpha_1}}\le\limsup_n\frac{\J(u_n)}{\|u_n\|_X^{\mu-\alpha_1}}\\
&\le\limsup_n\left(\frac2q(a_1+a_2) -\int_{\R^N}\frac{G(x, u_n(x))}{\|u_n\|_X^{\mu-\alpha_1}} dx\right)\\
&\le \frac2q(a_1+a_2) -\int_{\R^N}\liminf_n\frac{G(x, u_n(x))}{|u_n(x)|^{\mu-\alpha_1}}|v_n(x)|^{\mu-\alpha_1} dx.
\end{split}
\]
Then
\[
\int_{\R^N}\liminf_n\frac{G(x, u_n(x))}{|u_n(x)|^{\mu-\alpha_1}}|v_n(x)|^{\mu-\alpha_1} dx\le\frac2q(a_1+a_2),
\]
which contradicts \eqref{G+inf} as $|A|>0$.
\end{proof}

\begin{proposition} \label{gerho}
Assume that $(h_0)$--$(h_2)$, $(g_0)$--$(g_1)$, $(g_3)$ and $(P_1)$ hold. Then two positive constants $\rho$, $\beta > 0$ exist so that
\begin{equation}\label{geoN}
\J(u)\ge\beta \quad\mbox{ for all $u\in X$ with $\|u\|_V=\rho$}.
\end{equation}
\end{proposition}

\begin{proof}
Let $u\in X$ and, from $(g_3)$, take $\varepsilon >0$ such that 
$\bar{\alpha} + \varepsilon <\frac{\alpha_0}{\tau_{p, V}^p}$,
i.e.,
\begin{equation}\label{geoN1}
\alpha_0 - (\bar{\alpha}+\varepsilon) \tau_{p, V}^p > 0.
\end{equation}
Then from $(h_2)$ with $\alpha_0 \le 1$ (see Remark \ref{RemV1}),  \eqref{funct}, \eqref{Glimcom} and \eqref{sob} we have that
\[
\begin{split}
\J(u)&\ge\frac{\alpha_0}{p}\int_{\R^N}(|\nabla u|^p dx +V(x)|u|^p) dx +\frac{\alpha_0}{q}\int_{\R^N}(|\nabla u|^q dx +W(x)|u|^q) dx \\
&-\frac{\bar{\alpha}+\varepsilon}{p}\int_{\R^N}|u|^p dx -\frac{c_{\varepsilon}}{s}\int_{\R^N} |u|^s dx\\
&\ge\ \frac{1}{p}\ \big(\alpha_0 - (\bar{\alpha}+\varepsilon) \tau_{p, V}^p\big) \|u\|_V^p 
- \frac{c_{\varepsilon}}{s} \tau_{s, V}^s \|u\|_V^s.
\end{split}
\]
Then \eqref{geoN1} and $p<s$
allow us to find two positive constants $\rho$ and $\beta$ 
so that \eqref{geoN} holds.
\end{proof}

Moreover,  in order to prove Theorem \ref{MoltMain}, the following result is needed, too.
\begin{proposition}\label{PropMolt}
Suppose that assumptions $(h_0)$--$(h_2)$, $(g_0)$--$(g_1)$ and $(P_1)$--$(P_2)$ hold. Fixing any $\beta>0$, then $j\in\N$ and $\varrho$ sufficiently large exist such that
\[
\J(u)\ge\beta \quad\mbox{ for all $u\in Z^j$ with $\|u\|_V=\varrho$},
\]
with $Z^j$ as in \eqref{Zjdefn}.
\end{proposition}
\begin{proof}
Let $\beta>0$ be fixed.  Taking $u\in Z^j$, $j\in\N$, from \eqref{funct}, hypothesis $(h_2)$, \eqref{Gle}, \eqref{sob}, \eqref{2Zineq} and Hölder's inequality we deduce that
\[
\begin{split}
\J(u)&\ge\frac1p\left(\alpha_0- \frac{a}{\eta_{j+1}}\right)\|u\|_V^p + \frac{\alpha_0}{q}\|u\|_W^q-\frac{a}{s}\int_{\R^N}|u|^s dx\\
&\ge \frac1p\left(\alpha_0- \frac{a}{\eta_{j+1}}-\frac{a\tau_{p^*, V}^{s-r}}{s\eta_{j+1}^{\frac{r}{p}}}\|u\|_V^{s-p}\right)\|u\|_V^p
\end{split}
\]
where $r$ is taken in such a way that $\frac{r}{p}+\frac{s-r}{p^*}=1$.  By \eqref{inc}, we can always choose $\varrho$ large enough such that
\[
\frac{\alpha_0}{2p}\varrho^p \;>\; \beta
\]
and $j$ large so that
\[
\alpha_0- \frac{a}{\eta_{j+1}}-\frac{a\tau_{p^*, V}^{s-r}}{s\eta_{j+1}^{\frac{r}{p}}}\varrho^{s-p}>\frac{\alpha_0}{2},
\]
whence we obtain the desired result.
\end{proof}

Lastly, we provide a convergence result in the whole space $\R^N$.

\begin{proposition}\label{convRN}
Suppose that hypotheses $(P_1)$--$(P_2)$, $(h_0)$--$(h_3)$, $(g_0)$--$(g_2)$ hold. 
Taking any $\gamma\in\R$, we have that any $(CPS)_{\gamma}$--sequence $(u_n)_n\subset X$ 
is bounded in $W^{1, p}_V(\R^N)\cap W^{1, q}_{W}(\R^N)$. 
Furthermore, $u\in W^{1, p}_V(\R^N)\cap W^{1, q}_{W}(\R^N)$ exists such that, up to subsequences, the following limits hold
\begin{align}   \label{conv1}
&u_n\rightharpoonup u \mbox{ weakly in } W_V^{1, p}(\R^N)\cap W^{1, q}_{W}(\R^N),\\    
\label{conv2}
&u_n\to u \mbox{ strongly in } L^r(\R^N) \mbox{ for each } r\in [p, p^*[,\\     
\label{conv3}
&u_n\to u \mbox{ a.e. in } \R^N.
\end{align}
\end{proposition}

\begin{proof}
Taking $\gamma\in\R$, let $(u_n)_n\subset X$ be a $(CPS)_{\gamma}$--sequence of $\J$, i.e.,
\begin{equation}  \label{beta0}
\J(u_n)\to\gamma\quad\mbox{ and }\quad \|d\J(u_n)\|_{X^{\prime}}(1+\|u_n\|_X)\to 0 \qquad\mbox{ as } n\to +\infty.
\end{equation}
Thus, from \eqref{funct}, \eqref{diff}, \eqref{beta0}, 
assumptions $(h_2)$--$(h_3)$, $(g_2)$ and also \eqref{weightnorm},
direct computations imply that
\[
\begin{split}
&\mu\gamma +\varepsilon_n =\mu\J(u_n)-\left\langle d\J(u_n), u_n\right\rangle
\ge \frac{\alpha_0\alpha_1}{p}\int_{\R^N}|\nabla u_n|^p dx\\
&+\frac{\mu-p}{p}\int_{\R^N} V(x)|u_n|^p dx + \frac{\alpha_0\alpha_1}{q}\int_{\R^N}|\nabla u_n|^q dx +\frac{\mu-q}{q}\int_{\R^N} W(x) |u_n|^q dx.
\end{split}
\]
It follows, in particular, that $(u_n)_n$ is bounded in the reflexive Banach space $W_V^{1, p}(\R^N)\cap W_W^{1, q}(\R^N)$ and then we have \eqref{conv1}.  Since
\[
W_V^{1, p}(\R^N)\cap W_W^{1, q}(\R^N)\hookrightarrow W_V^{1, p}(\R^N),
\]
we obtain \eqref{conv2} and \eqref{conv3} by using \eqref{comp} with $l=p$ and $U=V$.
\end{proof}

\section{Changeover through bounded domains}\label{sec_bounded}
From now on, let $\Omega$ denote an open bounded domain in $\R^N$. Thus, we define
\begin{equation}\label{Xlim}
X_{\Omega} = W_{0,V}^{1, p}(\Omega)\cap W_{0,W}^{1, q}(\Omega)\cap  L^{\infty}(\Omega)
\end{equation}
endowed with the norm
\[
\|u\|_{X_{\Omega}} =\|u\|_{\Omega,V} +\|u\|_{\Omega, W} +|u|_{\Omega,\infty} \quad\mbox{ for any } u\in X_{\Omega}
\]
where, matching with the notations used in \eqref{Xnorm2}, we set
\[
\|\cdot\|_{\Omega, V} = \|\cdot\|_{\Omega, V, p} \quad\mbox{ and }\quad \|\cdot\|_{\Omega, W} =\|\cdot\|_{\Omega, W, q}.
\]
We denote by $X_{\Omega}^{\prime}$ the dual space of $X$. 
\begin{remark} \label{Rmbdd}
As $\Omega$ is a bounded domain, 
not only $\|u\|_{\Omega, p}$ and $|\nabla u|_{\Omega,p}$ are equivalent norms  
but also, if assumption $(P_3)$ holds, a constant $c_{\Omega} \ge 1$ exists such that
\[
\begin{split}
\|u\|_{\Omega,V, p}^p &= \int_{\Omega} |\nabla u|^p dx + \int_{\Omega} V(x) |u|^p dx\\
&\le \int_{\Omega} |\nabla u|^p dx +c_{\Omega} \int_{\Omega}|u|^p dx\le c_\Omega \|u\|_{\Omega, p}^p.
\end{split}
\]
Hence, Remark \ref{Remb} implies that 
the norms $\|\cdot\|_{\Omega,V, p}$ and $\|\cdot\|_{\Omega, p}$ are equivalent, too.  The same arguments apply for the norms $\|\cdot\|_{\Omega, W, q}$ and $\|\cdot\|_{\Omega, q}$.\\
Moreover, since definition \eqref{Xlim} holds with $1<q \le p$ and $\Omega$ is a bounded domain,  from now on we refer to $X_{\Omega}$ as 
\[
X_{\Omega} =W_{0, V}^{1, p}(\Omega)\cap L^{\infty}(\Omega)
\]
which can be equipped with the norm
\[
\|u\|_{X_{\Omega}} =\|u\|_{\Omega,p} +|u|_{\Omega,\infty} \quad\mbox{ for any } u\in X_{\Omega}.
\]
\end{remark}
Actually, since any function $u \in  X_{\Omega}$ can be trivially extended 
to a function $\tilde{u}\in X$ just assuming $\tilde{u}(x) = 0$ for all $x \in \R^N \setminus \Omega$, 
then
\begin{equation}\label{Xequ}
\begin{split}
\|\tilde{u}\|_V &= \|u\|_{\Omega,V}, 
\quad \|\tilde{u}\|_{W} = \|u\|_{\Omega,W},\\
|\tilde{u}|_{\infty} &= |u|_{\Omega,\infty},\quad \ 
\|\tilde{u}\|_X = \|u\|_{X_\Omega}.
\end{split}
\end{equation}
In this setting,  \eqref{G0=0} and \eqref{funct} imply that the restriction of the functional 
$\J$ to $X_\Omega$, namely 
\[
\J_\Omega = \J|_{X_{\Omega}},
\]
is such that
\begin{equation}  \label{funOm}
\begin{split}
\J_{\Omega}(u) &= \frac1p \int_{\Omega} A(x, u)|\nabla u|^p dx +\frac1q \int_{\Omega}B(x, u)|\nabla u|^q dx\\
&\quad+ \frac1p\int_{\Omega} V(x) |u|^p dx +\frac1q\int_{\Omega} W(x)|u|^q dx-\int_{\Omega} G(x, u) dx.
\end{split}
\end{equation}
\begin{remark}
We note that, setting
\begin{equation}\label{gtilde}
\tilde{g}(x,t) = g(x,t) - V(x) |t|^{p-2} t -W(x)|t|^{q-2} t
\end{equation}
for a.e.  $x \in\R^N$ all $t\in\R$, from $(g_0)$ and $(P_1)$ we have that $\tilde{g}$ is a $\mathcal{C}^0$--Caratheodory function
such that
\begin{equation}\label{gtilde1}
\tilde{G}(x,t) = \int_0^t \tilde{g}(x,s) ds = 
G(x,t)- \frac1p V(x) |t|^{p} -\frac1q W(x)|t|^q
\end{equation}
for a.e.  $x \in\R^N$, all $t\in\R$. Then, by means of assumptions $(g_1)$ and $(P_3)$, some positive constants $\tilde{a}_1,  \tilde{a}_2$ depending on $\Omega$ exist such that
\begin{equation}\label{gtildeup}
|\tilde{g}(x, t)|\le \tilde{a}_1 |t|^{s-1} +\tilde{a}_2 |t|^{q-1}
\quad\mbox{ for a.e. } x \in\Omega, \; \mbox{ all $t\in\R$.}
\end{equation}
In particular, from \eqref{funOm} and \eqref{gtilde1} the functional $\J_{\Omega}$ can be written as
\begin{equation}  \label{funOm1}
\begin{split}
\J_{\Omega}(u) = \ &\frac1p \int_{\Omega} A(x, u)|\nabla u|^p dx +\frac1q \int_{\Omega} B(x, u) |\nabla u|^q dx\\
&- \int_{\Omega} \tilde{G}(x,u) dx, \quad u\in X_{\Omega}.
\end{split}
\end{equation}
\end{remark}
Hence, as \eqref{gtildeup} holds, arguing as in \cite[Proposition 3.1]{CP2}, 
it follows that $\J_{\Omega}:X_{\Omega}\to\R$ is a 
$\mathcal{C}^1$ functional and for any $u$, $v\in X_{\Omega}$, 
its  Fréchet differential in $u$ along the direction $v$ is given by
\begin{equation}\label{diffOm}
\begin{split}
&\langle d\J_{\Omega}(u), v\rangle =\int_{\Omega} A(x, u)|\nabla u|^{p-2} \nabla u\cdot\nabla v dx +\frac1p \int_{\Omega} A_t(x, u) v|\nabla u|^p dx \\
&\qquad+\int_{\Omega} B(x, u)|\nabla u|^{q-2} \nabla u\cdot\nabla v dx +\frac1q \int_{\Omega} B_t(x, u) v|\nabla u|^q dx\\
&\qquad -\int_{\Omega} \tilde{g}(x, u) v dx.
\end{split}
\end{equation}

We prove that the functional $\J_\Omega$ verifies Definition \ref{wCPS}.\\
For simplicity, here and in the following we denote by $(\varepsilon_n)_n$ any sequence decreasing to $0$.
\begin{proposition} \label{wCPSbdd}
Under assumptions $(P_1)$--$(P_3)$, $(h_0)$--$(h_4)$ and $(g_0)$--$(g_2)$, the functional defined in \eqref{funOm} satisfies the $(wCPS)$ condition in $\R$.
\end{proposition}
\begin{proof}
Taking $\gamma\in \R$, let $(u_n)_n \subset X_{\Omega}$ be a $(CPS)_\gamma$--sequence, i.e.,
\begin{equation}\label{c1}
\J_{\Omega}(u_n) \to \gamma \quad \hbox{and}\quad \|d\J_{\Omega}(u_n)\|_{X'_{\Omega}}(1 + \|u_n\|_{X_{\Omega}}) \to 0\qquad
\mbox{if $\ n\to+\infty$.}
\end{equation}
We want to prove that $u \in X_{\Omega}$ exists such that 
\begin{itemize}
\item[$(i)$] $\ \|u_n - u\|_{\Omega, p} \to 0$ (up to subsequences), 
\item[$(ii)$] $\ \J_{\Omega}(u) = \gamma$, $\;d\J_{\Omega}(u) = 0$.
\end{itemize}
The proof will be divided in the following steps.
\begin{itemize}
\item[1.] $(u_n)_n$ is bounded in $W_0^{1, p}(\Omega)$; hence, up to subsequences, $u \in W^{1,p}_{0}(\Omega)$ exists such that if $n\to+\infty$, then
\begin{eqnarray}
&&u_n \rightharpoonup u\ \hbox{weakly in $W^{1,p}_{0}(\Omega)$,}
\label{c2}\\
&&u_n \to u\ \hbox{strongly in $L^r(\Omega)$ for each $r \in [1,p^*[$,}
\label{c3}\\
&&u_n \to u\ \hbox{a.e. in $\Omega$}.
\label{c4}
\end{eqnarray}
\item[2.] $u \in L^\infty(\Omega)$.
\item[3.] Having defined $T_k : \R \to \R$ such that
\[
T_k t = \left\{\begin{array}{ll}
t&\hbox{if $|t| \le k$}\\
k \frac t{|t|}&\hbox{if $|t| > k$}
\end{array}\right. ,
\]
if $k \ge |u|_{\infty,\Omega} + 1$ then
\[
\|d\J_{\Omega}(T_k u_n)\|_{X'_{\Omega}} \to 0\quad \hbox{and}\quad
\J_{\Omega}(T_k u_n)\to\gamma \qquad \hbox{as $n \to +\infty$.}
\]
\item[4.] $\|T_k u_n - u\|_{\Omega, p} \to 0$ if $n\to+\infty$; hence, $(i)$ holds.
\item[5.] $(ii)$ is satisfied.
\end{itemize}
\emph{Step 1.} Let $\gamma\in\R$ be fixed and consider a sequence $(u_n)_n\subset X$ satisfying \eqref{c1}.
Thus from Remark \ref{Rmbdd} and same arguments as in Proposition \ref{convRN}, we infer that
\[
\begin{split}
\mu\gamma +\varepsilon_n \ge d_1\|u_n\|^p_{\Omega, p} +d_2\|u_n\|^q_{\Omega,q}.
\end{split}
\]
It follows that $(u_n)_n$ is bounded in $W_0^{1, p}(\Omega)$ and then $u\in W_0^{1, p}(\Omega)$ exists so that, up to subsequences, \eqref{c2}--\eqref{c4} hold.\\
\emph{Step 2.} If $u\notin L^{\infty}(\Omega)$,  either
\begin{equation}    \label{sup_u}
\esssup_{\Omega} u = +\infty
\end{equation}
or
\begin{equation}     \label{sup_menou}
\esssup_{\Omega}(-u) = +\infty. 
\end{equation}
For example, suppose that \eqref{sup_u} holds.
Then, for any fixed $k\in\N$, we have
\begin{align}         \label{mpos}
|\Omega_{k}^{+}|> 0,
\end{align}
with $\Omega_{k}^{+}:= \left\lbrace x\in\Omega \mid u(x) > k\right\rbrace$.
Now, we consider the new function $R^{+}_k: t\in\R \mapsto R^{+}_k t\in\R$ such that
\[
R^{+}_k t = \begin{cases}
0 &\hbox{ if } t\leq k\\
t - k &\hbox{ if } t > k
\end{cases}.
\]
From \eqref{c2}, we have
\[
R^{+}_k u_n \rightharpoonup R^{+}_k u \quad \mbox{ in } W^{1, p}_{0}(\Omega).
\]
Thus by the sequentially weakly lower semicontinuity of $\Vert\cdot\Vert_{\Omega, p}$, it follows that
\[
\int_{\Omega} \vert\nabla R^{+}_k u\vert^{p} dx 
\leq \liminf_{n\to +\infty} \int_{\Omega} \vert\nabla R^{+}_k u_n\vert^{p} dx
\]
i.e.,
\begin{equation}        \label{3.23}
\int_{\Omega_{k}^{+}} \vert\nabla u\vert^{p} dx
 \leq \liminf_{n\to +\infty}\int_{\Omega_{n, k}^{+}} \vert\nabla u_n\vert^{p} dx,
\end{equation}
with $\Omega_{n, k}^{+}:=\left\lbrace x\in\Omega\mid u_n(x) > k\right\rbrace$.
On the other hand, by definition, we have
\[
\Vert R^{+}_k u_n\Vert_{X_{\Omega}} \leq \Vert u_n\Vert_{X_{\Omega}}.
\]
Hence, \eqref{c1} implies that
\[
|\langle d\J_{\Omega}(u_n), R^+_k u_n\rangle|\to 0;
\]
thus, from \eqref{mpos} an integer $n_k\in\N$ exists such that
\begin{equation}      \label{m+mpos}
|\langle d\J_{\Omega}(u_n), R^+_k u_n\rangle|
 < |\Omega^{+}_{k}| \quad \mbox{ for all } n\geq n_k.
\end{equation}
We observe that, taking any $k \ge 1$ and $n \in \N$, from \eqref{gtilde}, \eqref{diffOm}, hypotheses $(P_1)$, $(h_2)$ and $(h_4)$ with $\alpha_2 <q<p$, we obtain
\[
\begin{split}
&\langle d\J_{\Omega}(u_n), R^+_k u_n\rangle=\int_{\Omega^+_{n, k}} \left(1-\frac{k}{u_n}\right)\left(A(x, u_n)+\frac1p A_t(x, u_n) u_n \right) |\nabla u_n|^p dx\\
&\qquad+\int_{\Omega^+_{n, k}} \frac{k}{u_n} A(x, u_n)|\nabla u_n|^p dx+\int_{\Omega^+_{n, k}} \left(1-\frac{k}{u_n}\right) V(x) |u_n|^p dx\\
&\qquad+\int_{\Omega^+_{n, k}} \left(1-\frac{k}{u_n}\right)\left(B(x, u_n)+\frac1q B_t(x, u_n) u_n \right) |\nabla u_n|^q dx\\
&\qquad+\int_{\Omega^+_{n, k}} \frac{k}{u_n} B(x, u_n)|\nabla u_n|^q dx +\int_{\Omega^+_{n, k}} \left(1-\frac{k}{u_n}\right) W(x) |u_n|^q dx\\
&\qquad -\int_{\Omega} g(x, u_n) R^+_k u_n dx\\
&\qquad\ge \frac{\alpha_0\alpha_2}{p}\int_{\Omega^+_{n, k}} |\nabla u_n|^p dx +\frac{\alpha_0\alpha_2}{q}\int_{\Omega^+_{n, k}} |\nabla u_n|^q dx -\int_{\Omega} g(x, u_n) R^+_k u_n dx\\
&\qquad\ge \frac{\alpha_0\alpha_2}{p}\int_{\Omega^+_{n, k}} |\nabla u_n|^p -\int_{\Omega} g(x, u_n) R^+_k u_n dx
\end{split}
\]
which, together with \eqref{m+mpos}, implies
\begin{equation}    \label{mu1lambda}
\frac{\alpha_0\alpha_2}{p} \int_{\Omega^{+}_{n,k}} \vert\nabla u_n\vert^{p} dx 
\leq |\Omega^{+}_{k}| + \int_{\Omega} g(x, u_n) R^{+}_k u_n dx
\quad \mbox{ for all } n\geq n_k.
\end{equation}
Now, from \eqref{c4} and hypothesis $(g_0)$, we have
\[
g(x, u_n) R^+_k u_n\to g(x, u) R^+_k u \quad\mbox{ a.e. in } \Omega.
\]
Furthermore, since $|R^+_k u_n|\le |u_n|$ for a.e. $x\in\Omega$, by using assumption $(g_1)$ and \eqref{c3}, we infer that $h\in L^1(\Omega)$ exists so that
\[
|g(x, u_n)R^+_k u_n|\le d_3(|u_n|^s +|u_n|^p)\le h(x) \quad\mbox{ for a.e. } x\in\Omega
\]
see \cite[Theorem 4.9]{Brezis}. Thus, Dominated Convergence Theorem applies and we have
\begin{equation}\label{limg}
\lim_{n\to +\infty}\int_{\Omega} g(x, u_n) R^+_k u_n dx =\int_{\Omega} g(x, u)R^+_k u dx.
\end{equation}
Now,  again hypothesis $(g_1)$ and \eqref{hpsotto}, \eqref{3.23}, \eqref{mu1lambda}, \eqref{limg} ensure that
\[
\int_{\Omega_k^+} |\nabla u|^p dx \le d_4\left(|\Omega^+_k| + \int_{\Omega^+_k} |u|^s dx\right).
\]
As this inequality holds for any $k\ge 1$, thus \cite[Lemma 4.5]{CP2} applies and then $\displaystyle\esssup_{\Omega} u$ is bounded, yielding a contradiction to \eqref{sup_u}. \\
Suppose that \eqref{sup_menou} holds; it implies that,
fixing any $k\in\N$, we have
\[
|\Omega^{-}_k|>0 \quad \hbox{with}\; \Omega^{-}_k = \lbrace x\in\Omega: u(x) < -k\rbrace.
\]
In this case, by replacing function $R^{+}_k$ with $R^{-}_k:t\in\R\mapsto R^{-}_k t\in\R$ such that
\[
R^{-}_k t:=
\begin{cases}
0 &\hbox{ if } t\geq -k\\
t+k &\hbox{ if } t<-k
\end{cases},
\]
we can reason as above so to apply again \cite[Lemma 4.5]{CP2} which yields a contradiction 
to \eqref{sup_menou}. Then, it has to be $u\in L^{\infty}(\Omega)$.\\
The remaining \emph{Steps 3, 4, 5} can be addressed arguing as in \cite[Proposition 4.6]{CP2}, just taking $\tilde{g}$ as in \eqref{gtilde} and considering the functional as in \eqref{funOm1}.
\end{proof}

From now on, assume that $0\in\Omega$. Then $\varepsilon^*>0$ can be found so that $B_{\varepsilon^*}(0)\subset\Omega$.
\begin{remark} \label{Rmk*}
Suppose that hypotheses of Proposition \ref{geo1} hold. Thus fixing $\bar{u}\in X\setminus\{0\}$ such that ${\rm supp}\bar{u}\subset B_{\varepsilon^*}(0)$, a sufficiently large constant $\bar{\sigma}$ can be found such that
\begin{equation} \label{minbeta}
\|u^*\|_{\Omega, V}\ge\rho\quad\mbox{ and }\quad \J_{\Omega}(u^*)<\beta,
\end{equation}
where $u^*=\bar{\sigma} \bar{u}$ and $\rho, \beta$ are as in Proposition \ref{gerho}.
\end{remark}
\begin{proposition} \label{ExistOm}
Under the assumptions of Theorem \ref{ExistMain}, there exists $\beta>0$ such that the functional $\J_{\Omega}$ admits at least a critical point $u_{\Omega}\in X_{\Omega}$ satisfying
\[
\beta\le\J_{\Omega}(u_{\Omega})\le\sup_{\sigma\in [0,1]}\J_{\Omega}(\sigma u^*).
\]
\end{proposition}
\begin{proof}
We observe that $\J_{\Omega}$ is a $\mathcal{C}^1$ functional satisfying the weak Cerami--Palais--Smale condition (see Propositions \ref{C1} and \ref{wCPSbdd}) and by \eqref{G0=0} we have $\J_{\Omega}(0)=0$. Furthermore,  from \eqref{Xequ}, \eqref{funOm} and Proposition \ref{gerho} we infer the existence of $\rho, \beta>0$ which do not depend on $\Omega$ and satisfy
\[
\J(u)\ge\beta \quad\mbox{ for all $u\in X_{\Omega}$, $\|u\|_{\Omega, V}=\rho$}.
\]
Taking $u^*$ as in Remark \ref{Rmk*}, from \eqref{minbeta} we have that Theorem \ref{absMPT} applies to the functional $\J_{\Omega}$ in the Banach space $X_{\Omega}\hookrightarrow W_{0, V}^{1, p}(\Omega)$. Thus, a critical point $u_{\Omega}\in X_{\Omega}$ exists such that
\[
\J_{\Omega}(u_{\Omega})=\inf_{\gamma\in\Gamma_{\Omega}}\sup_{\sigma\in [0, 1]}\J_{\Omega}(\gamma(\sigma))\ge\beta,
\]
with $\Gamma_{\Omega}=\{\gamma\in C([0, 1], X_{\Omega}): \gamma(0)=0, \gamma(1) =u^*\}$. Since we assumed $B_{\varepsilon^*}(0)\subset\Omega$, the claimed result follows by noting that the function
\[
\gamma^*:\sigma\in[0, 1]\mapsto \sigma u^*\in X_{\Omega}
\]
belongs to $\Gamma_{\Omega}$.
\end{proof}
\begin{proposition} \label{MoltOm}
Let the assumptions of Theorem \ref{MoltMain} hold. Thus, for any $\beta>0$, the functional $\J_{\Omega}$ admits at least a pair of symmetric critical points in $X_{\Omega}$, whose corresponding critical level belongs to $[\beta, \beta^*]$, with $\beta^*$ independent of $\Omega$.
\end{proposition}
\begin{proof}
As already mentioned in Proposition \ref{ExistOm}, the functional $\J_{\Omega}$ is of class $\mathcal{C}^1$ and satisfies the weak--Cerami--Palais--Smale condition.
Moreover, as $g(x, \cdot)$ is odd,  it follows from \eqref{funOm} that $\J_{\Omega}$ is even.  
Now fixing $\beta>0$, Propositions \ref{geo1} and \ref{PropMolt} guarantee that Theorem \ref{SMPT} applies. Hence a pair $u_{\Omega}, -u_{\Omega}$ of critical points of $\J_{\Omega}$ exists such that
\[
\beta\le \J_{\Omega}(u_{\Omega})\le \sup_{Y^j} \J(u),
\]
where $Y^j$ is a finite dimensional subspace of $X$ such that \eqref{oplus} holds, with $j$ and $Z^j$ as in Proposition \ref{PropMolt}. We note that $\displaystyle\sup_{Y^j} \J(u)$ is a real constant independent of $\Omega$.
\end{proof}
\section{Proof of Theorems \ref{ExistMain} and \ref{MoltMain}}\label{sec_proof}

Throughout this section, for each $k\in\N$ we consider the open ball $B_k$, its related
Banach space $X_{B_k}$ as in \eqref{Xlim} and the functional 
\[
\J_k : u \in X_{B_k} \mapsto \J_k(u) = \J_{B_k}(u) \in \R
\]
with $\J_{B_k}=\J|_{B_k}$. 
\begin{remark}\label{rem00}
For the sake of convenience, if $u \in X_{B_k}$ we always consider its trivial
extension to $0$ in $\R^N \setminus B_k$. If we still denote such an 
extension by $u$ then we have that $u \in X$, too. Moreover, 
from \eqref{G0=0}, Definitions \eqref{funct} and \eqref{funOm}, 
respectively \eqref{diff} and \eqref{diffOm}, imply that
$\J_k(u) = \J(u)$, respectively
\[
\langle d\J_k(u),v\rangle = \langle d\J(u),v\rangle
\quad \hbox{for all $v \in X_{B_k}$.}
\]
\end{remark}
Suppose that the assumptions in Theorem \ref{ExistMain} are satisfied. Since for all $k \in \N$ Proposition \ref{ExistOm}
applies to $\J_k$ in $X_{B_k}$, from Remark \ref{rem00}
a sequence $(u_k)_k \subset X$ exists such that 
for every $k \in \N$ we have
\begin{itemize}
\item[$(i)$] $\ u_k \in X_{B_k}$ with $u_k = 0$ in $\R^N \setminus B_k$,
\item[$(ii)$] $\ \displaystyle \beta\ \le\ \J(u_k)\ \le \ \sup_{\sigma\in [0,1]}\J(\sigma u^*)=: \beta^*$,
\item[$(iii)$] $\langle d\J(u_k),v\rangle = 0$ for all $v \in X_{B_k}$,
\end{itemize}
where $\beta$ and $\beta^*$ are real positive constants which does not depend on $k$.\\
The proof of Theorem \ref{ExistMain} requires different steps.
Firstly, we prove that sequence $(u_k)_k$
is bounded in $X$. To this aim, the following result is needed (see \cite[Lemma 5.6]{CSS2}).

\begin{lemma}
\label{tecnico} 
Let $\Omega$ be an open bounded domain in $\R^N$ and
consider $p$, $r$ so that  $1 < p \le N$
and $p \le r < p^*$ (if $N = p$ we just require that $p^*$ is any
number larger than $r$)
and take $u \in W_0^{1,p}(\Omega)$. If $a^* >0$ and $m_0\in \N$
exist such that 
\[
\label{LUuno}
\int_{\Omega^+_m}|\nabla u|^p dx \le a^* \left(m^r\ |\Omega^+_m| +
\int_{\Omega^+_m} |u|^r dx\right)\quad\hbox{for all $m \ge m_0$,}
\]
with $\Omega^+_m = \{x \in \Omega:\ u(x) > m\}$, then $\displaystyle \esssup_{\Omega} u$
is bounded from above by a positive constant which can be chosen depending only on $|\Omega^+_{m_0}|$, $N$, $p$, $r$, $a^*$, $m_0$, 
$\|u\|_{\Omega, p}$, or better by a positive constant which can be chosen so 
that it depends only on $N$, $p$, $r$, $a^*$, $m_0$ and $a_0^*$
for any $a_0^* > 0$ such that 
\[
\max\{|\Omega^+_{m_0}|,\ \|u\|_{\Omega, p}\}\ \le\ a_0^*.
\]
Vice versa, if inequality
\[
\int_{\Omega^-_m}|\nabla u|^p dx \le a^*\left(m^r\ |\Omega^-_m| +
\int_{\Omega^-_m} |u|^r dx\right) \quad\hbox{for all $m \ge m_0$,}
\]
holds with $\Omega^-_m = \{x \in \Omega:\ u(x) < - m\}$, 
then $\displaystyle \esssup_{\Omega}(-u)$ 
is bounded from above by a positive constant which can be
chosen so that it depends only on $N$, $p$, $r$, $a^*$, $m_0$ and any
constant which is greater than both $|\Omega^-_{m_0}|$ and 
$\|u\|_{\Omega, p}$.
\end{lemma}
\begin{proposition} \label{boundX}
A constant $M_0 > 0$ exists such that
\begin{equation}\label{bddX}
\|u_k\|_{X}\le M_0\quad\mbox{ for all } k\in\N.
\end{equation}
\end{proposition}

\begin{proof}
From conditions $(i)$ and $(iii)$ we infer that
\begin{equation}\label{p0}
\left\langle d\J(u_k), u_k\right\rangle =0 \quad\mbox{ for all } k\in\N.
\end{equation}
Thus from \eqref{p0}, taking $\mu$ as in assumption $(h_3)$ and arguing as in Proposition \ref{convRN} we get
\[
\mu\J(u_k)\ =\ \mu\J(u_k)-\left\langle d\J(u_k), u_k\right\rangle\ge d_1\|u_k\|_{V}^p +d_2\|u_k\|_W^q,
\]
which, combined with condition $(ii)$,  ensures that
\begin{equation}\label{M1}
\|u_k\|_{V} +\|u_k\|_W\le d_3 \quad\mbox{ for all } k\in\N.
\end{equation}
In particular, from Remark \ref{Remb}, \eqref{Xnorm2} and \eqref{M1} we get
\begin{equation} \label{ukBk}
\|u_k\|_{B_k, p}\le d_4 \quad\mbox{ for all } k\in\N.
\end{equation}
Now, in order to obtain estimate \eqref{bddX}, from
definition \eqref{Xdefn} 
we need to prove that $(u_k)_k$ is a bounded sequence in $L^{\infty}(\R^N)$ too.
To this aim, we note that
for a fixed $k \in \N$ either $|u_k|_\infty \le 1$ or $\ |u_k|_\infty > 1$.
If $\ |u_k|_\infty > 1$, then 
\[
\esssup_{\R^N} u_k > 1\quad \hbox{and/or} \quad
\esssup_{\R^N} (-u_k) > 1.
\]
Assume that $\displaystyle \esssup_{\R^N} u_k > 1$
and consider the set
\[
B^{+}_{k,1} = \{x \in \R^N:\ u_k(x) > 1\}.
\]
From condition $(i)$ we have that
\[
B^{+}_{k,1} \subset \ B_k,
\]
$B^{+}_{k,1}$ is an open bounded domain
such that not only $|B^{+}_{k,1}| > 0$ but also, by means of \eqref{WVem} we have
\[
\begin{split}
|B^{+}_{k,1}| &\le \int_{B^{+}_{k,1}} (|u_k|^p +|u_k|^q) dx\le \int_{\R^N} (|u_k|^p +|u_k|^q) dx\\
& \le \|u_k\|_p^p+\|u_k\|_q^q\le d_5 (\|u_k\|_V^p +\|u_k\|_W^q).
\end{split}
\]
Hence estimate \eqref{M1} gives 
\begin{equation}\label{measbdd}
|B^{+}_{k,1}|\le d_6.
\end{equation}
For any $m \in \N$ we consider 
the new function $R^+_m : \R \to \R$ defined as
\[
t \mapsto R^+_mt = \left\{\begin{array}{ll}
0&\hbox{if $t \le m$}\\
t-m&\hbox{if $t > m$}
\end{array}\right. .
\]
Clearly, for any $m \ge 1$, 
we have that $R^+_m u_k \in X_k$ and from condition $(iii)$,  \eqref{diffOm} and hypotheses $(V_1)$, $(h_2)$, $(h_4)$ and same arguments as in Step 2 of Proposition \ref{wCPSbdd}, it follows that
\[
0\ =\ \langle d\J(u_k), R^+_m u_k\rangle \ge \frac{\alpha_0\alpha_2}{p}\int_{B^+_{k,m}} |\nabla u_k|^p -\int_{\R^N} g(x, u_k) R^+_m u_k dx,
\]
i.e.,
\begin{equation}\label{ie}
\frac{\alpha_0\alpha_2}{p}\int_{B^+_{k,m}} |\nabla u_k|^p dx
\le \int_{\R^N} g(x, u_k)R^+_m u_k dx,
\end{equation}
with $B^+_{k,m} = \{x\in\R^N: u_k(x)>m\}$. 
Clearly, we have $B^+_{k,m}\subseteq B^+_{k,1}$ so from
\eqref{measbdd} we have that
\begin{equation}   \label{OmM}
|B^+_{k,m}|\le d_7.
\end{equation}
Since $(g_2)$ implies $g(x,t) > 0$ if $t>0$ for a.e. $x \in \R^N$, then
$g(x,u_k(x)) > 0$ for a.e. $x \in B_{k,m}^{+}$; so, hypothesis $(g_1)$ implies that
\[
\int_{\R^N} g(x, u_k)R^+_m u_k dx \le \int_{B^{+}_{k,m}} g(x, u_k)u_k dx\le 2a\int_{B^{+}_{k,m}} u_k^s dx,
\]
which, together with \eqref{ie}, suggest
\[
\int_{B^{+}_{k,m}} |\nabla u_k|^p dx \le  
d_8 \int_{B^{+}_{k,m}} u_k^s dx \quad \hbox{for all $m \ge 1$,}
\]
with $d_8> 0$ independent of $m$ and $k$. Since \eqref{hpsotto} holds, Lemma \ref{tecnico}  
applies with $\Omega=B_{k}$ and $d_9 > 1$ exists such that
\[
\esssup_{B_{k}} u_k \le d_9,
\]
where the constant $d_9$ can be choosen independent of $k$ as it only rely upon $N, p, s, d_8$ and $a_0^*$, with $a_0^*\ge \max\{d_7, d_4\}$ (see \eqref{ukBk} and \eqref{OmM}).\\
Similar arguments apply if $\displaystyle \esssup_{\R^N} (-u_k) > 1$.
Summing up, $d_{10} \ge 1$ exists such that
\[
|u_k|_{\infty} \le d_{10}\qquad \hbox{for all $k \in \N$}
\]
and \eqref{bddX} is complied. 
\end{proof}

From estimate \eqref{bddX} and \eqref{Xdefn} it follows that $u_\infty \in W_V^{1, p}(\R^N)\cap W_W^{1, q}(\R^N)$
exists such that, up to subsequences, 
\begin{equation} \label{weakV}
u_k\rightharpoonup u_\infty \quad\mbox{ weakly in } W_V^{1,p}(\R^N)\cap W_W^{1, q}(\R^N).
\end{equation}
Furthermore, from $ii)$ in Corollary \ref{Cor1} it follows that, up to subsequences,
\begin{align} \label{al1}
&u_k\to u_{\infty} \quad\mbox{strongly in } L^r(\R^N) \mbox{ for each } r\ge p;\\ \label{aeRN}
&u_k\to u_{\infty} \quad\mbox{a.e. in }\; \R^N.
\end{align}
On the other hand,  from \eqref{bddX} and \eqref{aeRN}, we have
\begin{equation} \label{inLinf}
u_{\infty}\in L^{\infty}(\R^N)
\end{equation}
(see \cite[Proposition 6.5]{CSS2} for the details). Hence, by definition \eqref{Xdefn}, $u_{\infty}\in X$. 

\begin{proposition}\label{limwpr}
We have that
\begin{equation}\label{strongwpr}
u_k \to u_\infty \quad \hbox{strongly in $W_V^{1,p}(\R^N)\cap W_W^{1, q}(\R^N)$.}
\end{equation}
\end{proposition}

\begin{proof}
Following an idea introduced in \cite{BMP}, 
let us consider the real map $\psi(t) = t \e^{\eta t^2}$, 
where $\eta > (\frac{\beta_2}{2\beta_1})^2$ will be fixed once
$\beta_1$, $\beta_2 > 0$ are chosen in a suitable way later on. By
definition, we have that
\begin{equation}\label{eq4}
\beta_1 \psi'(t) - \beta_2 |\psi(t)| > \frac{\beta_1} 2\qquad \hbox{for all $t \in \R$.}
\end{equation}
Defining $v_{k}=u_k - u_\infty$, \eqref{weakV} implies that
\begin{equation}\label{cc2}
v_{k} \rightharpoonup 0\quad \hbox{weakly in $W^{1,p}_V(\R^N)\cap W^{1, q}_W (\R^N)$,}
\end{equation}
while from \eqref{al1}, respectively \eqref{aeRN},
it follows that
\begin{equation}\label{cc21}
v_{k} \to 0 \quad \hbox{strongly in $L^r(\R^N)$ for each $r\ge p$;} 
\end{equation}
respectively
\begin{equation}\label{cc22}
v_{k} \to 0 \quad\hbox{a.e. in $\R^N$.}
\end{equation}
Moreover, from \eqref{inLinf} and \eqref{bddX} we obtain that
\begin{equation}\label{cc23}
v_k \in X\quad \hbox{and}\quad |v_{k}|_\infty \le \bar{M}_0 \quad\hbox{for all $k \in\N$,}
\end{equation}
with $\bar{M}_0 = M_0 + |u_\infty|_\infty$.
Hence, from \eqref{cc23} we have
\begin{equation}\label{stim1}
|\psi(v_{k})| \le \psi(\bar{M}_0),\quad 0<\psi'(v_{k}) \le \psi'(\bar{M}_0) \qquad\hbox{a.e. in $\R^N$,}
\end{equation}
while \eqref{cc22} implies
\begin{equation}\label{stim2}
\psi(v_{k}) \to 0, \quad
\psi'(v_{k}) \to 1 \qquad\hbox{a.e. in $\R^N$.}
\end{equation}
Thus, since $\psi(v_k)\in X_{B_k}$, from $(iii)$ and \eqref{diff} we infer that
\begin{equation}  \label{6.40}
\begin{split}
0\ = &\left\langle d\J(u_k), \psi(v_k)\right\rangle
=\int_{\R^N} \psi^{\prime}(v_k) A(x, u_k) |\nabla u_k|^{p-2} \nabla u_k\cdot\nabla v_k dx\\
&+\int_{\R^N} \psi^{\prime}(v_k) B(x, u_k) |\nabla u_k|^{q-2} \nabla u_k\cdot\nabla v_k dx\\
& +\frac1p \int_{\R^N} A_t(x, u_k) \psi(v_k) |\nabla u_k|^p dx\\
& +\frac1q \int_{\R^N} B_t(x, u_k) \psi(v_k) |\nabla u_k|^q dx +\int_{\R^N} V(x)|u_k|^{p-2} u_k \psi(v_k) dx\\
& +\int_{\R^N} W(x)|u_k|^{q-2} u_k \psi(v_k) dx -\int_{\R^N} g(x, u_k) \psi(v_k) dx.
\end{split}
\end{equation}
We note that $(g_0)$--$(g_1)$ together with \eqref{stim1},  Hölder inequality, \eqref{cont1} \eqref{bddX}, 
 imply that
\[
\left\vert\int_{\R^N} g(x, u_k) \psi(v_k) dx\right\vert \le d_1(|u_k|_p^{p-1} |v_k|_p^p +|u_k|_s^{s-1}|v_k|_s^s)\le d_2(|v_k|_p^p +|v_k|_s^s)
\]
which, together with \eqref{cc21} ensures that
\begin{equation}\label{ip2}
\int_{\R^N} g(x,u_k) \psi(v_k) dx \to 0.
\end{equation}
Moreover \eqref{bddX}, assumptions $(h_1)$--$(h_2)$ 
and direct computations give
\[
\begin{split}
&\left\vert\int_{\R^N}A_t(x, u_k) \psi(v_{k}) |\nabla u_k|^p dx\right\vert
\le\ d_3\int_{\R^N} |\psi(v_{k})| |\nabla u_k|^{p} dx\\
&\qquad=d_3\Big(\int_{\R^N} |\psi(v_{k})| |\nabla u_k|^{p-2}\nabla u_k\cdot\nabla v_{k} dx\\
&\qquad\qquad\qquad\qquad\quad +\int_{\R^N}|\psi(v_{k})| |\nabla u_k|^{p-2}\nabla u_k\cdot\nabla u_\infty dx\Big)\\
&\qquad= d_3\int_{\R^N} |\psi(v_{k})| |\nabla u_k|^{p-2}\nabla u_k\cdot\nabla v_{k} dx +\varepsilon_k
\end{split}
\]
since by the Hölder inequality and \eqref{bddX} we have
\begin{equation} \label{tre3}
\begin{split}
&\int_{\R^N} |\psi(v_{k})| |\nabla u_k|^{p-2}\nabla u_k\cdot\nabla u_\infty dx\\
&\quad\le d_4\left(\int_{\R^N}|\psi(v_{k})|^p |\nabla u_\infty|^p dx\right)^{\frac1p}
\left(\int_{\R^N} |\nabla u_k|^p dx\right)^{\frac{p-1}{p}} \\
&\quad\le d_5\left(\int_{\R^N}|\psi(v_{k})|^p |\nabla u_\infty|^p dx\right)^{\frac1p}\
\to\ 0
\end{split}
\end{equation}
as \eqref{stim1}, \eqref{stim2} allow us to apply the Dominated Convergence Theorem which ensures that
\[
\int_{\R^N}|\psi(v_{k})|^p |\nabla u_\infty|^p dx\to 0.
\]
Same arguments can be rearranged to prove that
\begin{equation}   \label{Bt6}
\begin{split}
&\left\vert\int_{\R^N}B_t(x, u_k) \psi(v_{k}) |\nabla u_k|^q dx\right\vert\\
&\qquad\le d_6\int_{\R^N} |\psi(v_{k})| |\nabla u_k|^{q-2}\nabla u_k\cdot\nabla v_{k} dx +\varepsilon_k.
\end{split}
\end{equation}
Back to \eqref{6.40}, using estimates \eqref{ip2}--\eqref{Bt6} and setting
\[
h_k(x) =\psi^{\prime}(v_k(x))-d_7 |\psi(v_k(x))| \quad\mbox{ for a.e. } x\in\R^N
\]
with $d_7 =\max\left\{\frac{d_3}{p}, \frac{d_6}{q}\right\}$, we obtain that
\begin{equation}\label{epsquasi}
\begin{split}
\varepsilon_k &\ge \int_{\R^N} A(x, u_k) h_k(x) |\nabla u_k|^{p-2}\nabla u_k\cdot\nabla v_k dx\\
&\quad +\int_{\R^N} B(x, u_k) h_k(x) |\nabla u_k|^{q-2}\nabla u_k\cdot\nabla v_k dx\\
&\quad+\int_{\R^N} V(x)|u_k|^{p-2} u_k \psi(v_k) dx +\int_{\R^N} W(x)|u_k|^{q-2} u_k \psi(v_k) dx.
\end{split}
\end{equation}
On the other hand,  taking $\beta_1=1$ and $\beta_2 =d_7$, 
from \eqref{stim1}, \eqref{stim2} and direct computations we have that
\begin{equation} \label{stima10}
h_{k}(x) \to 1 \ \hbox{a.e. in $\R^N$} \ 
\hbox{and} \
|h_{k}(x)| \le \psi'(\bar{M}_0) + d_7|\psi(\bar{M}_0)|\ \hbox{a.e. in $\R^N$,}
\end{equation}
while from \eqref{eq4} we deduce
\begin{equation} \label{frac12}
h_{k}(x)>\frac12\quad\mbox{ a.e. in $\R^N$}.
\end{equation}
Back to \eqref{epsquasi}, direct computations give
\begin{equation}\label{quasi2}
\begin{split} 
\varepsilon_{k}\ \ge\ & +\int_{\R^N}(h_k A(x, u_k) -A(x, u_{\infty}))|\nabla u_{\infty}|^{p-2}\nabla u_{\infty}\cdot\nabla v_k \,dx\\
&+\int_{\R^N} h_k A(x, u_k)(|\nabla u_k|^{p-2}\nabla u_k -|\nabla u_{\infty}|^{p-2}\nabla u_{\infty})\cdot\nabla v_k\, dx\\
& +\int_{\R^N}(h_k B(x, u_k) -B(x, u_{\infty}))|\nabla u_{\infty}|^{q-2}\nabla u_{\infty}\cdot\nabla v_k\, dx\\
&+\int_{\R^N} h_k B(x, u_k)(|\nabla u_k|^{q-2}\nabla u_k -|\nabla u_{\infty}|^{q-2}\nabla u_{\infty})\cdot\nabla v_k\, dx\\
&+\int_{\R^N} V(x)(|u_k|^{p-2} u_k -|u_{\infty}|^{p-2} u_{\infty})\psi(v_k) dx\\
& +\int_{\R^N} V(x) |u_{\infty}|^{p-2} u_{\infty} \psi(v_k)\,dx\\
&+\int_{\R^N} W(x)(|u_k|^{q-2} u_k -|u_{\infty}|^{q-2} u_{\infty})\psi(v_k) dx\\
& +\int_{\R^N} W(x) |u_{\infty}|^{q-2} u_{\infty} \psi(v_k)\,dx
\end{split}
\end{equation}
where we have used \eqref{cc2} to infer that
\[ 
\begin{split}
\int_{\R^N} A(x,u_\infty) |\nabla u_\infty|^{p-2}\nabla u_\infty \cdot\nabla v_{k}dx\ \to\ 0,\\[3pt]
 \int_{\R^N} B(x,u_\infty) |\nabla u_\infty|^{q-2}\nabla u_\infty \cdot\nabla v_{k}dx\ \to\ 0.
\end{split}
\]
On the other hand, from Hölder's inequality, we obtain
\begin{equation} \label{hkDC}
\begin{split}
&\int_{\R^N} |(h_k A(x, u_k) -A(x,u_\infty))|\nabla u_\infty|^{p-2}\nabla u_\infty\cdot\nabla v_{k}| dx\\
&\quad\le\left(\int_{\R^N}|h_{k} A(x, u_k) -A(x,u_\infty)|^{\frac{p}{p-1}}|\nabla u_\infty|^p dx\right)^{\frac{p-1}{p}}
\|v_{k}\|_p\ \to\ 0,
\end{split}
\end{equation}
since assumption $(h_0)$, \eqref{aeRN} and \eqref{stima10} imply that 
\[
h_{k}A(x,u_k) -A(x,u_\infty) \to 0 \quad \hbox{a.e. in $\R^N$,}
\]
while \eqref{bddX}, \eqref{stima10}, \eqref{inLinf} and $(h_1)$ give 
\[
|h_{k} A(x, u_k) -A(x,u_\infty)|^{\frac{p}{p-1}}|\nabla u_\infty|^p\le d_8 |\nabla u_\infty|^p
\quad \hbox{a.e. in $\R^N$.}
\]
Thus \eqref{hkDC} follows by \eqref{cc21} and the Dominated Convergence Theorem.\\
Moreover, we have
\begin{equation}\label{four}
\begin{split}
&\int_{\R^N} V(x) |u_{\infty}|^{p-2} u_{\infty} \psi(v_k)dx\to 0\\
& \int_{\R^N} W(x) |u_{\infty}|^{q-2} u_{\infty} \psi(v_k)dx\to 0.
\end{split}
\end{equation}
In fact, via the identification
\[
\varphi\in L_V^p(\R^N)\mapsto\int_{\R^N} V(x)|u_{\infty}|^{p-2} u_{\infty} \, \varphi\, dx\in\R
\]
and Hölder's inequality, we obtain that
\[
\begin{split}
\left\vert \int_{\R^N} V(x)|u_{\infty}|^{p-2} u_{\infty} \, \varphi\, dx\right\vert &\le\left(\int_{\R^N} V(x)|\varphi|^p dx\right)^{\frac1p}\left(\int_{\R^N} V(x) |u_{\infty}|^p dx\right)^{\frac{p-1}{p}}\\[3pt]
&\le d_9 \|\varphi\|_{V}^{p-1},
\end{split}
\]
i.e. $V(x)|u_{\infty}|^{p-2} u_{\infty} \in (L_V^p(\R^N))^{\prime}$. Since the previous argument can be rephrased to prove that $W(x)|u_{\infty}|^{q-2} u_{\infty} \in (L_W^q(\R^N))^{\prime}$, we deduce from \eqref{cc2} that \eqref{four} holds.\\
Finally, summming up, from \eqref{frac12}--\eqref{four}, assumption $(P_1)$, the definition of $\psi(v_k)$ and strong convexity of the power function with exponent $p, q>1$, we have
\[
\begin{split}
\varepsilon_k &\ge \frac{\alpha_0}{2}\int_{\R^N} \big[(|\nabla u_k|^{p-2}\nabla u_k -|\nabla u_\infty|^{p-2}\nabla u_\infty)\\
&\qquad\qquad\qquad +(|\nabla u_k|^{q-2}\nabla u_k -|\nabla u_\infty|^{q-2}\nabla u_\infty)\big]\cdot\nabla v_{k} dx\\
&+\int_{\R^N} V(x)(|u_k|^{p-2} u_k -|u_{\infty}|^{p-2} u_{\infty}) v_k dx\\
&\qquad\qquad\qquad +\int_{\R^N} W(x)(|u_k|^{q-2} u_k -|u_{\infty}|^{q-2} u_{\infty}) v_k dx\ge 0.
\end{split}
\]
Hence we infer that both
\[
\begin{split}
&\int_{\R^N} (|\nabla u_k|^{p-2}\nabla u_k -|\nabla u_\infty|^{p-2}\nabla u_\infty)\cdot\nabla(u_k-u_{\infty}) dx\to 0,\\
&\int_{\R^N}(|\nabla u_k|^{q-2}\nabla u_k 
-|\nabla u_\infty|^{q-2}\nabla u_\infty)\cdot\nabla (u_k -u_{\infty}) dx\to 0
\end{split}
\]
and
\[
\begin{split}
&\int_{\R^N} V(x)(|u_k|^{p-2} u_k -|u_{\infty}|^{p-2} u_{\infty})(u_k-u_{\infty}) dx \to 0, \\
&\int_{\R^N} W(x)(|u_k|^{q-2} u_k -|u_{\infty}|^{q-2} u_{\infty})(u_k-u_{\infty}) dx\to 0.
\end{split}
\]
Thus \eqref{strongwpr} holds.
\end{proof}

\begin{proposition}\label{critt}
One has
\begin{equation}\label{crit0}
\langle d\J(u_\infty),\varphi\rangle = 0 \quad \hbox{for all $\varphi \in C_c^\infty(\R^N)$}
\end{equation}
with $C_c^\infty(\R^N) = \{\varphi \in C^\infty(\R^N):\ \supp\varphi\subset\subset \R^N\}$.
Hence, $d\J(u_\infty)= 0$ in $X$.  
\end{proposition}

\begin{proof}
Let $\varphi \in C_c^\infty(\Omega)$. Thus, a radius $R \ge 1$ exists
such that $\supp \varphi \subset B_R$. For all $k\ge R$
we have that $\varphi \in X_k$ so $(iii)$ holds and
$\langle d\J(u_k),\varphi\rangle = 0$. 
Moreover, \eqref{bddX}, \eqref{aeRN}, Propositions \ref{limwpr} and \ref{C1}, ensure that
\[
\|d\J(u_k)-d\J(u_{\infty})\|_{X^{\prime}}\to 0\quad\mbox{ as $k\to +\infty$,}
\]
hence \eqref{crit0} holds.
\end{proof}

\begin{proof}[Proof of Theorem \ref{ExistMain}]
Thanks to Proposition \ref{critt}, the statement of Theorem \ref{ExistMain} is true 
if we prove $u_\infty \not\equiv 0$.\\ 
Suppose by contradiction that $u_\infty=0$. Thus, from assumption $(g_1)$ and \eqref{al1} 
we have that
\begin{equation}\label{gto0}
\int_{\R^N} g(x, u_k) u_k dx\ \to\ 0,
\end{equation}
which, together with assumption $(g_2)$, ensures that
\begin{equation}\label{Gto0}
\int_{\R^N} G(x, u_k) dx \to 0.
\end{equation}
Furthermore, from \eqref{diff}, $(iii)$, $(h_4)$ and also $(h_2)$ it results
\[
\begin{split}
0 &= \left\langle d\J(u_k), u_k\right\rangle =\int_{\R^N} A(x, u_k)|\nabla u_k|^p dx
 +\frac1p\int_{\R^N} A_t(x, u_k) u_k |\nabla u_k|^p dx\\
&\quad +\int_{\R^N} V(x) |u_k|^p dx +\int_{\R^N} B(x, u_k)|\nabla u_k|^q dx\\
&\quad +\frac1q\int_{\R^N} B_t(x, u_k) u_k |\nabla u_k|^q dx
 +\int_{\R^N} W(x) |u_k|^q dx-\int_{\R^N} g(x, u_k) u_k dx\\
& \ge \frac{\alpha_2\alpha_0}{p}\int_{\R^N} |\nabla u_k|^p dx+ \int_{\R^N} V(x) |u_k|^p dx + \frac{\alpha_2\alpha_0}{q}\int_{\R^N} |\nabla u_k|^q dx\\
&\quad +\int_{\R^N} W(x) |u_k|^q dx-\int_{\R^N} g(x, u_k) u_k dx,
\end{split}
\]
which implies that
\begin{equation} \label{weigh0}
\|u_k\|_V^p +\|u_k\|_W^q \to 0\quad\mbox{ as } k\to +\infty
\end{equation}
by means of \eqref{weightnorm} and \eqref{gto0}.
Hence, from \eqref{funct}, \eqref{Gto0},  \eqref{bddX}, assumption $(h_1)$,  and \eqref{weigh0} we infer 
\[
\J(u_k) \le d_1 (\|u_k\|_V^p +\|u_k\|_W^q) - \int_{\R^N} G(x, u_k) dx \to 0
\]
in contradiction with the estimate in $(ii)$. 
\end{proof}
\begin{proof}[Proof of Theorem \ref{MoltMain}]
Fix $\beta>0$. For any $k\in\N$, let $u_k$ be the critical point established via Proposition \ref{MoltOm}. Thus, the sequence $(u_k)_k\subset X$ meets conditions $(i)$--$(iii)$ and, in particular
\[
\beta\le \J(u_k)\le \beta^*,
\]
where $\beta^*$ is independent of $k$. Arguing as done above,  it follows that $u_{\infty}\in X$ exists which is a critical point for the functional $\J$. Clearly, Proposition \ref{C1} implies
\[
\J(u_k)\to \J(u_{\infty}),
\]
hence $\J(u_{\infty})\ge\beta$. Summing up, for any $\beta>0$, a critical point $u_{\infty}\in X$ exists with critical level greater or equal to $\beta$.
By the arbitrariness of $\beta$,  we obtain the claimed result.
\end{proof}

\bigbreak

\noindent{\bf \large Funding}

\medskip

\noindent The research that led to the present paper was partially supported 
by MIUR--PRIN project ``Qualitative and quantitative aspects of nonlinear PDEs'' (2017JPCAPN\underline{\ }005) and {\sl Fondi di Ricerca di Ateneo} 2017/18 ``Problemi differenziali non lineari''.\\
Both the authors are members of the Research Group INdAM--GNAMPA.

\end{document}